\documentclass[a4paper]{amsart}
\usepackage{amsmath}
\usepackage{amssymb}
\newcommand{\un}{{\rm 1\kern -0.4em 1}}
\newcommand{\Ni}{\mathbb{N}}

\newcommand{\CC}{\mathbb{C}}

\newcommand{\NN}{\mathbb{N}}
\newcommand{\RR}{\mathbb{R}}

\newcommand{\UU}{\mathbb{U}}
\newcommand{\ZZ}{\mathbb{Z}}

\newtheorem{Theorem}{Theorem}[section]
\newtheorem{Proposition}[Theorem]{Proposition}
\newtheorem{Lemma}[Theorem]{Lemma}
\newtheorem{Corollary}[Theorem]{Corollary}
\newtheorem{Remark}[Theorem]{Remark}
\newtheorem{Fact}[Theorem]{Fact}
\newtheorem{Remarks}[Theorem]{Remarks}
\newtheorem{defin}[Theorem]{Definition}
\newenvironment{defi}{\begin{defin} \rm}{\end{defin}}
\newtheorem{defins}[Theorem]{Definitions}
\newenvironment{defis}{\begin{defins} \rm}{\end{defins}}
\newtheorem{example}[Theorem]{Example}
\newenvironment{examp}{\begin{example} \rm}{\end{example}}
\newtheorem{examples}[Theorem]{Examples}
\newenvironment{examps}{\begin{examples} \rm}{\end{examples}}
\newtheorem{notat}[Theorem]{Notation}

\newtheorem{notats}[Theorem]{Notations}
\newenvironment{notas}{\begin{notats} \rm}{\end{notats}}
\title{MV-algebras and partially cyclically ordered groups.}
\author{G. Leloup}
\date{December 12, 2018}
\address{G.\ Leloup\\ 
Laboratoire Manceau de Math\'ematiques\\
Facult\'e des Sciences\\
avenue Olivier Messiaen\\
72085 LE MANS CEDEX\\
FRANCE}
\email{gerard.leloup@univ-lemans.fr}
\begin{document}
\begin{abstract} 
We prove that there exists a functorial correspondence between MV-algebras and partially 
cyclically ordered groups which are wound round of lattice-ordered groups. 
It follows that some results about cyclically ordered groups can be 
stated in terms of MV-algebras. For example, the study of groups together with 
a cyclic order allows to get a first-order characterization of groups of unimodular complex 
numbers and of finite cyclic groups. We deduce a characterization of pseudofinite MV-chains 
and of  pseudo-simple MV-chains (i.e.\ which share the same first-order properties as some 
simple ones). We can generalize these results to some non-lineraly ordered 
MV-algebras, for example hyper-archimedean MV-algebras. 
\end{abstract}
\maketitle
\section{Introduction.}
\indent 
This article has been written in such a way that it can be read by someone who does not have prior knowledge 
of cyclically ordered groups. We list all the definitions and properties that we need. We also list all the definitions and properties about MV-algebras and logic that we need. 
We try to bring back whenever possible to properties partially ordered groups, lattice-ordered groups or 
linearly ordered groups. \\
\indent Unless otherwise stated the groups are abelian groups.\\ 
\indent 
Every MV-algebra can be obtained in the following way. Let $(G,\leq ,\wedge,u)$ be a 
lattice-ordered group (in short $\ell$-group) 
together with a distinguished strong unit $u>0$ (i.e.\ for every $x\in G$ there 
is a positive 
integer $n$ such that $x\leq nu$); such a group will be called a unital $\ell$-group. 
We set $[0,u]:=\{x\in G\mid 0\leq x \leq u\}$. For every 
$x$, $y$ in $[0,u]$ we let $x\oplus y=(x+y)\wedge u$ and $\neg x=u-x$. We see that the 
restriction of the partial order $\leq$ to $[0,u]$ can be defined by $x\leq y\Leftrightarrow 
\exists z\; y=x\oplus z$. \\ 
\indent 
Now, the quotient group $C=G/\ZZ u$ can be equipped with a partial cyclic order. 
First, we explain what is a cyclic order. On a circle $C$, there is no canonical linear order, but there 
exists a canonical cyclic order. Assume that one traverses a circle counterclockwise. 
We set that $R(x,y,z)$ holds if one can find $x$, $y$, $z$ in this order starting from some point of 
the circle. Now, starting from another point one can find them in the order $y$, $z$, $x$ or 
$z$, $x$, $y$. So in turn $R(y,z,x)$ and $R(z,x,y)$ hold. We say that $R$ is cyclic. Furthermore, 
for any $x$ in $C$ the relation $y<_x z\Leftrightarrow R(x,y,z)$ is a relation of linear order 
on the set $C\backslash \{x\}$. \\[2mm]
\setlength{\unitlength}{1mm}
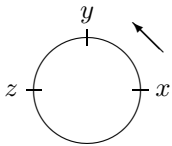
\begin{picture}(200,20)(0,0)
\put(20,10){\circle{14}}
\put(20,16){\line(0,1){2}}
\put(26,10){\line(1,0){2}}
\put(12,10){\line(1,0){2}}
\put(20,20){\makebox(0,0){$y$}}
\put(30,10){\makebox(0,0){$x$}}
\put(10,10){\makebox(0,0){$z$}}
\put(30,15){\vector(-1,1){4}}
\end{picture}
These rules give the definition of a cyclic order. We generalize this definition to a 
partial cyclic order by assuming that $<_x$ is a partial order relation which needs not be 
a linear order. \\ 
\indent  
Turning to the cyclic order $R(\cdot,\cdot,\cdot)$ on the quotient group $C=G/\ZZ u$ it is defined by setting 
$R(x_1+\ZZ u,x_2+\ZZ u,x_3+\ZZ u)$ if there exists 
$n_2$ and $n_3$ in $\ZZ$ such that $x_1<x_2+n_2u<x_3+n_3u<x_1+u$ 
(see Proposition \ref{prop37}). One can prove that 
$(C,R)$ satisfies for every $x$, $y$, $z$, $v$ in $C$: \\ 
\indent 
$R(x,y,z) \Rightarrow x \neq y \neq z \neq x$ ($R$ is strict)\\ 
\indent  
$R(x,y,z) \Rightarrow R(y,z,x)$ ($R$ is cyclic)\\ 
\indent  
by setting $y\leq_x z$ if either 
$R(x,y,z)$ or $x\neq y=z$ or $x=y\neq z$ or $x=y=z$, then  $\leq_x$ is a partial order relation on $C$. \\ 
\indent 
$R(x,y,z) \Rightarrow R(x+v,y+v,z+v)$ ($R$ is compatible). \\ 
\indent  
Any group equipped with a ternary relation which satisfies those properties is called a partially cyclically 
ordered group. If all the orders $\leq_x$ 
are linear orders, then $(C,R)$ is called a cyclically ordered group. In the case where 
$C=G/\ZZ u$, where $G$ is a partially ordered group, we say that $C$ is the wound-round of 
a partially ordered group. \\ 
\indent  
If $A$ is an MV-algebra, then there is a unital $\ell$-group $(G_A,u_A)$ 
(uniquely determined up to isomorphism) such that $A$ 
is isomorphic to the MV algebra $[0,u_A]$. The group $G_A$ is called the 
Chang $\ell$-group of $A$. 
We show that the MV-algebra $A$ is definable in the partially cyclically ordered group 
$(G_A/\ZZ u_A,R)$. It follows a functorial correspondence between cyclically ordered groups and 
MV-chains (i.e.\ linearly ordered MV-algebras). So some properties of MV chains can be 
deduced from analogous properties of cyclically ordered groups. \\ 
\indent 
In \cite{Glu 83}, D.\ Glusschankov constructed 
a functor between the category of projectable MV-algebras and the category of projectable 
lattice partially ordered groups. The approach of the present paper is different, and it 
does not need to restrict to a subclass of MV-algebras. \\[2mm]
\indent In Section \ref{section2} we list basic properties of MV-algebras and of their Chang 
$\ell$-groups. We also give a few basic notions of logic which we need in this paper. 
Section \ref{section3} is devoted to partially cyclically ordered groups. 
We focus on the set of elements which are non-isolated (i.e.\ the elements $x$ such that 
there exists $y$ satisfying $R(0,x,y)$ or $R(0,y,x)$) 
and on those partially cyclically ordered groups which can 
be seen as wound-rounds of lattice-ordered groups. We also list some results of \cite{LL 13} 
on c.o.\ groups which belong to the elementary class generated by the subgroups of the 
multiplicative group $\UU$ of unimodular complex numbers. From those properties 
it follows that there is a functor $\Theta\Xi$ from the category of MV-algebras to the category of 
partially cyclically ordered groups together with c-homorphisms. In Section \ref{section4} we define a class 
$\mathcal{AC}$ of partially cyclically ordered groups $C$ in wich we can define a MV-algebra $A(C)$ 
such that $C=\Theta\Xi (A(C))$ (Theorem \ref{n44}).Then we prove that the wound-rounds of 
$\ell$-groups belong to $\mathcal{AC}$ (Therorem \ref{prop321}). 
 Furthermore, the subgroup generated by $A(C)$ being 
the wound round of an $\ell$-group is expressible by finitely many first-order formulas 
(Theorem \ref{thm416}). 
In the case of MV-chains, the one-to-one mapping 
$C\mapsto A(C)$ defines a functorial correspondence between the class of cyclically ordered groups and the class of MV-chains. 
In section \ref{section5}  we prove that if $A$ and $A'$ are two MV-chains, then 
$A$ and $A'$ are elementarily equivalent if, and only if, $C(A)$ and $C(A')$ are elementarily equivalent 
(Proposition \ref{prop53}). 
We also prove that any 
two MV-chains are elementarily equivalent if, and only if, their Chang $\ell$-groups are 
elementarily equivalent (Proposition \ref{prop55}). 
The class of pseudo-simple MV-chains is defined to be the elementary class generated by the simple 
chains. We define in the same way the pseudofinite MV-chains. 
One can prove that a pseudo-simple MV-chain is an MV-chain which is elementarily equivalent to some 
MV-subchain of $\{x\in \RR\mid 0\leq x\leq 1\}$, and a pseudofinite MV-chain is an 
MV-chain which is elementarily equivalent to some ultraproduct of 
finite MV-chains. We use the results of 
\cite{LL 13} on cyclically ordered groups to deduce characterizations of pseudo-simple and of 
pseudofinite MV-chains. Furthermore, we get necessary and sufficient conditions for such 
MV-chains being elementarily equivalent (Theorems \ref{thm511}, \ref{th511a}, \ref{th511}). In 
Section \ref{section6}, we generalize the results of Section \ref{section5} about pseudofinite and pseudo-simple 
MV-chains to pseudo-finite and pseudo-hyperarchimedean MV-algebras which are cartesian products of finitely 
many MV-chain (Theorems \ref{th611}, \ref{thm611}). \\[2mm]
\indent The author would like to thank Daniele Mundici for his bibliographical advice and his suggestions. 
\section{MV-algebras.}\label{section2}
\indent  The reader can find more properties for example in \cite[Chapter 1]{CDM}. 
\subsection{Definitions and basic properties}
\begin{defi}\label{def21} 
An {\it MV-algebra} is a set $A$ equipped with a binary operation $\oplus$, 
a unary operation $\neg$ and a distinguished constant $0$ satisfying the following 
equations. For every $x$, $y$ and $z$: \\ 
MV1) $x\oplus (y\oplus z)=(x\oplus y)\oplus z$\\
MV2) $x\oplus y=y\oplus x$\\
MV3) $x\oplus 0=x$\\
MV4) $\neg\neg x=x$\\
MV5) $x\oplus \neg 0=\neg 0$\\
MV6) $\neg(\neg x\oplus y)\oplus y=\neg (\neg y\oplus x)\oplus x$. 
\end{defi}
\indent  If this holds, then we define $1=\neg 0$,  
$x\odot y=\neg(\neg x\oplus \neg y)$ and $x\leq y \Leftrightarrow 
\exists z,\; x\oplus z=y$. Then $\leq$ is a partial order called the 
{\it natural order} on $A$. This partial order satisfies: $x\leq y\Leftrightarrow 
\neg y\leq \neg x$, and $A$ is a lattice with smallest element $0$, greatest 
element $1$, where $x\vee y=\neg (\neg x\oplus y)\oplus y$, 
$x\wedge y=\neg(\neg(x\oplus \neg y)\oplus \neg y)=(x\oplus \neg y)\odot y$. 
The operations $\wedge$ (infimum) and $\vee$ (supremum) are compatible with $\oplus$ and $\odot$. 
Note that Condition MV6) can be written as $x\vee y=y\vee x$. 
In the case where $\leq$ is a linear order, 
$A$ is called an {\it MV-chain}. \\
\indent Following \cite{CDM}, if $n$ is a positive integer and $x$ is an element of a group, 
then we denote by $nx$ the sum $x+\cdots+x$ ($n$ times). If $x$ belongs to 
an MV-algebra, then we set $n.x=x\oplus\cdots\oplus x$ ($n$ times). Further,  
we will set $x^n=x\odot\cdots\odot x$. \\
\indent If $G$ is a partially ordered group and $u\in G$, $u$ is said to be 
a {\it strong unit} if $u>0$ and for every $x\in G$ there is $n\in \NN$ such that $x\leq nu$. 
It follows that there exists $n'\in \NN$ such that $-x\leq n'u$, hence 
$-n'u\leq x \leq nu$. A {\it unital $\ell$-group} is an $\ell$-group (i.e.\ a lattice-ordered group) 
with distinguished strong unit $u>0$. More generally, a {\it unital} partially (resp.\ linearly) ordered group 
is a partially (resp.\ linearly) ordered group together with a distinguished strong unit. 
\begin{examp} If $(G,u)$ is a unital $\ell$-group, then 
the set $[0,u]:=\{ x\in M\mid 0\leq x\leq u\}$ together with the operations 
$x\oplus y=(x+y)\wedge u$, $\neg 
x=u-x$, and where $0$ is the identity element of $G$, is an MV-algebra, whose natural partial order 
is the restriction of the partial order on $G$. It is denoted $\Gamma(G,u)$. 
In this case, $x\odot y=(x+y-u)\vee 0$, and it follows: $(x\oplus y)+(x\odot y)=x+y$. 
\end{examp}
\indent  
A {\it unital} homomorphism between two unital partially ordered groups $(G,u)$ and $(G',u')$ 
is an increasing group homomorphism $f$ between $G$ and $G'$ 
such that $f(u)=u'$. An $\ell$-homomorphism between two 
$\ell$-groups $G$ and $G'$ is a group-homomorphism such that for every 
$x$, $y$ in $G$ we have that $f(x\wedge y)=f(x)\wedge f(y)$ and $f(x\vee y)=
f(x)\vee f(y)$ (it follows that $f$ is also an increasing homomorphism). \\
\indent  A {\it homomorphism} of MV-algebras is a function $f$ from an MV-algebra $A$ to an 
MV-algebra $A'$ such that $f(0)=0$ and for every $x$, $y$ in $A$ $f(x\oplus y)=
f(x)\oplus f(y)$ and $f(\neg x)=\neg f(x)$. \\
\indent  The mapping $\Gamma: \; (G,u)\mapsto \Gamma(G,u)$ is a full and faithfull 
functor from the category $\mathcal{A}$ of unital $\ell$-groups 
to the category $\mathcal{MV}$ of MV-algebras. If $f$ is a unital $\ell$-homomorphism between $(G,u)$ and 
$(G',u')$, then $\Gamma(f)$ is the restriction of $f$ to $[0,u]$ (\cite[Chapter 7]{CDM}). \\
\indent  Now, for every MV-algebra $A$ there exists 
a unital $\ell$-group $(G_A,u_A)$, 
uniquely defined up to isomorphism, such that $A$ is isomorphic to $[0,u_A]$ 
together with above operations; $(G_A,u_A)$ is called the {\it Chang $\ell$-group} 
of $A$ (see \cite[Chapter 2]{CDM}). We will sometimes let $G_A$ stand for $(G_A,u_A)$. 
For further purposes we describe this Chang $\ell$-group at the end of 
this section. The mapping $\Xi$: $A\mapsto (G_A,u_A)$ 
is a full and faithfull functor from the category $\mathcal{MV}$ to the category $\mathcal{A}$. 
We do not describe here the unital $\ell$-homomorphism 
$\Xi(f)$, where $f$ is a homomorphism of MV-algebras. The composite functors 
$\Gamma\Xi$ and $\Xi\Gamma$ are naturally equivalent to the identities of the respective 
categories (see \cite{CDM} Theorems 7.1.2 and 7.1.7). \\[2mm]
\indent  The language of MV-algebras is $L_{MV}=(0,\oplus,\neg)$. However, since $\leq$, $\wedge$ and 
$\vee$ are definable in $L_{MV}$, we will assume that they belong to the language. 
We will denote by 
$L_o=(0,+,-,\leq)$ the language of ordered groups, by $L_{lo}=(0,+,-,\wedge,\vee)$ the 
language of $\ell$-groups. When dealing with unital $\ell$-groups, 
we will add a constant symbol to the language, it will be denoted by 
$L_{lou}=L_{lo}\cup \{ u\}$, and $L_{loZu}=L_{lo}\cup\{Zu\}$ will denote the language of 
$\ell$-groups together with a unary predicate for the subgroup generated by 
$u$. \\[2mm]
\indent  The Chang $\ell$-group $G_A$ of $A$ is an $L_{lou}$-structure where $u$ is a 
constant predicate interpreted by the distinguished strong unit of $G_A$. 
Now, $G_A$ is also a $L_{loZu}$-structure where $Zu$ is a unary predicate interpreted 
by: $Zu(x)$ if, and only if, $x$ belongs to the subgroup generated by the distinguished strong unit; in this 
case we denote it by $(G_A,\ZZ u_A)$. \\
\indent  In the MV-algebra $A$, recall that $x\odot y$ stands for $\neg(\neg x\oplus \neg y)$. 
In $[0,u_A]\subset G_A$ we have that $x\odot y=(x+y-u_A)\vee 0$. If $A$ is an MV-chain, then  
the formula $x>0$ and $0=x\odot x=\cdots =x^n$ is equivalent to: 
$0<x<2.x<\cdots <n.x\leq u_A$. 
\begin{notas} In the following, if $x<y$ are elements of a partially ordered group $G$, then we 
will set $[x,y]:=\{z\in G\mid x\leq z\leq y\}$, $[x,y[\; :=[x,y]\backslash \{y\}$, 
$]x,y]:=[x,y]\backslash \{x\}$ and $]x,y[\; :=[x,y]\backslash \{x,y\}$. In the particular case 
where $G=\RR$ (the group of real numbers) $x=0$ and $y=1$, then we let 
$[0,1]_{\RR}:=\{z\in \RR\mid 0\leq z\leq 1\}$. The notation $[0,1]$ will be used in the case 
of an MV-algebra.
\end{notas}
\begin{Lemma}\label{lm21} Let $G$ be an $\ell$-group and $0<u\in G$. \\ 
1) Either $[0,u]=\{0,u\}$, \\ 
or there exists $x\in\; ]0,u[$ such that $[0,u]=\{0,u,x\}$\\
or there exists $x\in\; ]0,u[$ such that $[0,u]=\{0,u,x,u-x\}$\\
or for every $x\in\; ]0,u[$ there exists $y\in\; ]0,u[$ such that $x<y$ or $y<x$. \\
2) If $]0,u[$ contains $x,y$ such that $x<y$, then for every $z\in\; ]0,u[$ there exists 
$z'\in\; ]0,u[$ such that $z<z'$ or $z'<z$.
\end{Lemma}
\begin{proof} 1) Assume that there exists $x\in\; ]0,u[$, 
and that $[0,u]\neq\{0,u,x\}$, $[0,u]\neq\{0,u,x,u-x\}$.\\
\indent  If $x=u-x$ (i.e.\ $2x=u$), then we let $z \in\; ]0,u[\;\backslash\{x\}$. 
By properties of $\ell$-groups, if $x\wedge z=0$, then for every $y>0$: 
$(x+y)\wedge z=y\wedge z$(see for example \cite[Lemma 2.3.4]{Gl 99}). 
Now, $(x+x)\wedge z=u\wedge z=z>0$, hence $x\wedge z>0$. 
If $x\wedge z < x$, then set $y=x\wedge z$. Otherwise $x<z$, hence we set $y=z$. \\
\indent  Now we assume that $x\neq u-x$. If $x<u-x$ or $u-x <x$, then we set $y=u-x$, 
otherwise, we have that $x\wedge (u-x)<x<x\vee (u-x)$. If $x\wedge (u-x) >0$, then we let 
$y=x\wedge (u-x)$. Otherwise, if $x \vee (u-x)\neq u$, then we let $y=x\vee (u-x)$. Assume 
that $x\wedge (u-x)=0$ and $x\vee (u-x)=u$, and let $z \in\; ]0,u[\;\backslash\{x,u-x\}$. 
If $x<z$, then let $y=z$. Otherwise, 
if $z\wedge x>0$, then we set $y=z\wedge x$. If $z\wedge x=0$, then $z=z\wedge u=
z\wedge (x\vee (u-x))=(z\wedge x)\vee (z\wedge (u-x))=z\wedge (u-x)$, hence 
$z<u-x$, and $x<u-z$, we let $y=u-z$. \\ 
\indent  2) Assume that $]0,u[$ contains $x,y$ such that $x>y$ (so it contains at least 
four elements). If $[0,u]=\{0,u,x,u-x\}$, then the result holds trivially ($y=u-x$). 
Otherwise $[0,u]$ contains at least five elements, and the result follows from 1). 
\end{proof}
\begin{Corollary}\label{cor22} Let $A$ be an MV-algebra. Then: \\ 
either $A=\{0,1\}$, \\ 
or there exists $x\in A\backslash \{0,1\}$ such that $A=\{0,1,x\}$\\
or there exists $x\in A\backslash \{0,1\}$ such that $A=\{0,1,x,\neg x\}$\\
or for every $x\in A\backslash \{0,1\}$ there exists $y\in A\backslash \{0,1\}$ 
such that $x<y$ or $y<x$. \\
 If $A\backslash \{0,1\}$ contains $x,y$ such that $x<y$, then for every 
$z\in A\backslash \{0,1\}$ there exists 
$z'\in A\backslash \{0,1\}$ such that $z<z'$ or $z'<z$.
\end{Corollary}

\subsection{Construction of the Chang $\ell$-group.}
The correspondence between MV-algebras and partially cyclically ordered groups relies on the 
good sequences defined for the construction of the Chang $\ell$-groups (see \cite{Mu 86}, 
\cite[Chapter 2]{CDM}). 
We describe this construction and we also define an analogue of the good sequences in an $\ell$-group. 
We start with some properties of partially ordered groups. 
\begin{Remark}\label{rk26} 
We know that every cancellative abelian monoid 
$M$ embeds canonically in a group $G$ generated by the image of $M$ 
following the construction of $\ZZ$ from $\NN$. Now, one can deduce from 
properties of partially ordered groups (see for example \cite{BKW}, Propositions 1.1.2, 1.1.3, and also 
1.2.5) that: \\
(i) $M$ is the 
positive cone of a compatible partial order on $G$ if, and only if, for every $x$, $y$ in $M$, $x+y=0
\Rightarrow x=y=0$, and this partial order is given by $x\leq y\Leftrightarrow \exists z\in M,\; y=x+z$, \\ 
(ii) $G$ is an $\ell$-group if, and only if, for every $x$, $y$ in $M$, $x\wedge y$ exists. 
\end{Remark}
\indent  The following lemmas show that every element of the positive cone of 
a unital $\ell$-group can be associated with a unique 
sequence of elements of $[0,u]$. So $G$ is determined by its restriction to 
$[0,u]$. This property will give rise to the construction of the Chang $\ell$-group. 
\begin{Lemma}\label{lm27} 
Assume that $(G,u)$ is a unital $\ell$-group.  
Let $0< x \in G$ and $m$ be a positive integer such that $x\leq mu$. 
Then, there exists a unique 
sequence $x_1,\dots,x_n$ of elements of $[0,u]$ such that $x=x_1+\cdots+x_n$ and, 
for $1\leq i<n-1$, $(u-x_i)\wedge(x_{i+1}+\cdots+x_n)=0$, and $n\leq m$.
\end{Lemma}
\begin{proof} For 
every $y\in G$, we have that $(u-y)\wedge(x-y)=0\Leftrightarrow (u\wedge x)-y=0\Leftrightarrow 
y=u\wedge x$. Set $x_1=x\wedge u$. Then $x_1$ is the unique element of $G$ 
such that $(u-x_1)\wedge (x-x_1)=0$. Since $0< x$, we have that $0\leq x_1\leq u$, and 
$0\leq x-x_1=x-(u\wedge x)=x+((-u)\vee (-x))=(x-u)\vee 0\leq (mu-u)\vee 0=(m-1)u$. 
By taking $x-x_1$ in place of $x$ we get $x_2\in [0,u]$ such that $(u-x_2)\wedge 
(x-x_1-x_2)=0$, and we have that $x-x_1-x_2\in [0,(m-2)u]$, and so on. Hence there exists a unique 
sequence $x_1,\dots,x_n$ of elements of $[0,u]$ such that $x=x_1+\cdots+x_n$ and, 
for $1\leq i<n-1$, $(u-x_i)\wedge(x_{i+1}+\cdots+x_n)=0$. 
\end{proof}
\begin{Lemma}\label{lem28} The condition: 
for $1\leq i<n-1$, $(u-x_i)\wedge(x_{i+1}+\cdots+x_n)=0$ is equivalent to: 
for $1\leq i<n-1$, $(u-x_i)\wedge x_{i+1}=0$. If this holds, then, for $1\leq i<j\leq n$, 
$(u-x_i)\wedge x_j=0$
\end{Lemma} 
\begin{proof} Assume that for 
$1\leq i<n-1$, $(u-x_i)\wedge(x_{i+1}+\cdots+x_n)=0$. Let $i<j\leq n$. Since $0\leq x_i$ and 
$0\leq x_j\leq x_{i+1}+\cdots+x_n$, it follows that $(u-x_i)\wedge x_j=0$. 
Now, let $y$, $z$, $z'$ in $[0,u]$ such that $(u-y)\wedge z=(u-z)\wedge z'=0$, then 
$0\leq (u-y)\wedge z' =(u-y)\wedge z'\wedge u=(u-y)\wedge z'\wedge (z+u-z)\leq 
((u-y)\wedge z'\wedge z)+((u-y)\wedge z' \wedge (u-z))=0$. Hence 
by induction we can prove that the condition: 
for $1\leq i<n-1$, $(u-x_i)\wedge x_{i+1}=0$ implies 
for $1\leq i<n-1$, $(u-x_i)\wedge(x_{i+1}+\cdots+x_n)=0$. 
\end{proof}
\begin{Remark}\label{rk210} 
By setting, for $x$, $y$ in $[0,u]$, $x\oplus y=(x+y)\wedge u$, we 
have that $x\odot y =0\vee (x+y-u)$. Hence, using the fact that for every $z$ in $G$ 
we have that $z=z\vee 0+z\wedge 0$, we get: $x+y=(x\oplus y)+ (x\odot y)$. 
We deduce from the proof of Lemma \ref{lm27} that 
$x\oplus y$ is the unique element of $[0,u]$ such that 
$(u-x\oplus y)\wedge (x+y-x\oplus y) =0$, and then $x+y-x\oplus y\in [0,u]$. It follows that 
$x=x\oplus y \Leftrightarrow (u-x)\wedge y=0$. Furthermore, from the equality 
$x+y=(x\oplus y)+(x\odot y)$ we get $x\oplus y=x\Leftrightarrow x\odot y=y$. \end{Remark}
\indent  Now, we come to the Chang $\ell$-group. 
\begin{defi} Let $A$ be an MV-algebra. 
A sequence $(x_i)$ of elements of $A$ indexed by the natural numbers $1,\; 2,\dots$ 
is said to be a {\it good sequence} if, for each $i$, $x_i\oplus x_{i+1}=x_i$, and it contains only 
a finite number of nonzero terms. If $x=(x_i)$ and $y=(y_i)$ are good sequences, then we define 
$z=x+y$ by the rules $z_1=x_1\oplus y_1$, $z_2=x_2\oplus (x_1\odot y_1)\oplus y_2$, and more generally, 
for every positive integer $i$: 
$$z_i=x_i\oplus(x_{i-1}\odot y_1)\oplus \cdots \oplus (x_1\odot y_{i-1})\oplus y_i.$$
We also define a partial order $\leq$ by $x\leq y\Leftrightarrow \exists z,\; y=x+z$. 
\end{defi}
\indent  We see that $A$ embeds into 
the monoid $M_A$ of good sequences by $x\mapsto (x,0,0,\dots)$, and 
one can prove that $M_A$ is cancellative, it satisfies the properties (i) and (ii) of Remark \ref{rk26}, 
where: \\
\indent  $x\wedge y=(x_i\wedge y_i)$, $x\vee y=(x_i\vee y_i)$, \\
\indent  if $y=x+z$, then $z=(y_i)+(\neg x_n,\neg x_{n-1}, 
\dots, \neg x_1,0,\dots)$, where $x_n$ is the last non-zero term of $(x_i)$ 
(see \cite[Chapter 2]{CDM}). \\
\indent  Consequently, $M_A$ defines in a unique way an $\ell$-group, and the image of 
$(1,0,\dots)$ in this $\ell$-group is a strong unit. This unital $\ell$-group is 
the Chang $\ell$-group $G_A$. 
\begin{Remark}\label{rk212} Let $x=(x_1,\dots,x_n,0,\dots)$ in the positive cone of $G_A$. 
Then $x=(x_1,0,\dots)+\cdots+(x_n,0,\dots)$. The embedding $x_i\mapsto (x_i,0,\dots)$ of 
$A$ in $G_A$ can be considered as an inclusion. Hence we can assume that $x_i\in G_A$ and write $x=x_1+\cdots+x_n$. 
Hence, the good sequence defining $x$ is the same as the sequence defined in Lemmas 
\ref{lm27} and \ref{lem28}. 
\end{Remark}
\begin{proof} Let $(x_i)$ be a good sequence, and for $k\geq 1$ let 
$y=(x_1,\dots,x_k,0,\dots)+(x_{k+1},0,\dots)$. By Lemma \ref{lem28} and Remark \ref{rk210}, we have, 
for $1\leq i<j\leq k+1$, $x_i\oplus x_j=x_i$ and $x_i\odot x_j=x_j$. It follows that \\
\indent  $y_1=x_1\oplus x_{k+1}=x_1$, \\
\indent  for $2\leq i\leq n$, $y_i=x_i\oplus(x_{i-1}\odot x_{k+1}1)\oplus \cdots \oplus (x_1\odot 0)\oplus 0=
x_i\oplus x_{k+1}=x_i$, \\
\indent  $y_{k+1}=0\oplus(x_n\odot x_{k+1})\oplus \cdots \oplus (x_1\odot 0)\oplus 0=0\oplus x_{k+1}\oplus 0=
x_{k+1}$, \\
and for $i>k+1$, $y_i=0$. \\ 
\indent  So by induction we get $(x_1,\dots,x_n,0,\dots)=(x_1,0\dots)+\cdots+(x_n,0,\dots)$. The remainder of 
the proof is straightforward. 
\end{proof}
\subsection{Elementary equivalence, interpretability.}\label{subsec21}
Two structures $S$ and $S'$ for a language $L$ are {\it elementarily 
equivalent} if any $L$-sentence is true in $S$ if, and only if, it is true in $S'$. We let 
$S\equiv S'$ stand for $S$ and $S'$ being elementarily equivalent. 
Furthermore if $S\subset S'$, then we say that $S$ is an {\it elementary substructure} of 
$S'$ (in short $S\prec S'$) 
if every existential formula with parameters in $S$ which is true in $S'$ is also true 
in $S$.
We will need the following properties. 
\begin{Theorem}\label{th26}
(\cite[Corollary 9.6.5 on p.\ 462]{Hod}, see also \cite[Theorems 5.1, 5.2]{FV 59}) 
Let $L$ be a first-order language. \\
(a) If $I$ is a non-empty set and for each $i\in I$, $A_i$ and $B_i$ are elementarily equivalent 
$L$-structures, then $\prod_IA_i\equiv \prod_I B_i$ (here $\prod$ denotes the direct product). \\
(b) If $I$ is a non-empty set and for each $i\in I$, $A_i$ and $B_i$ are L-structures with 
$A_i\prec B_i$, then $\prod_IA_i\prec \prod_I B_i$. \end{Theorem}
If, for every $i\in I$, $S_i$ is an $L$-structure and $U$ is an ultrafilter on $I$, then the 
{\it ultraproduct} of the $S_i$'s is the quotient set 
$\displaystyle{\left(\prod_{i\in I} S_i\right)/\sim}$, 
where $\sim$ is the equivalence relation: $(x_i)\sim (y_i) \Leftrightarrow 
\{i\in I\mid x_i=y_i \}\in U$. If $R$ is a unary predicate of $L$, then $R((x_i))$ holds in 
the ultraproduct if the set $\{i\in I\mid R(x_i) \mbox{ holds in } S_i \}$ belongs to $U$. 
Every relation symbol and every function symbol is interpreted in the same way. An {\it elementary 
class} is a class which is closed under ultraproducts. \\ 
\indent  A structure $S_1$ for a language $L_1$ is interpretable in 
a structure $S_2$ for a language $L_2$ if the following holds. \\ 
\indent  There is a one-to-one mapping $\varphi$ from a subset $T_1$ of $S_2$ onto $S_1$, \\
\indent  for every $L_1$-formula $\Phi$ of the form 
$R(\bar{x})$, $F(\bar{x})=y$, $x=y$ or $x=c$ (where $R$ is a relation symbol, $F$ is a 
function symbol and $c$ is a constant), there is an $L_2$-formula $\Phi'$ such that for every 
$\bar{x}$ in $T_1$, $S_1\models \Phi(\varphi(\bar{x})) \Leftrightarrow 
S_2\models \Phi'(\bar{x})$, \\ 
\indent  (this is a particular case of the definition of interpretability p.~58 and 
pp.~212-214 in \cite{Hod}). 
\begin{Theorem}\label{th27} (Reduction Theorem 5.3.2, \cite{Hod}). If $S_1$, $S_1'$ (resp.\ 
$S_2$, $S_2'$) are structures for the language $L_1$ (resp.\ $L_2$) such that 
$S_1$ is interpretable in $S_2$ and $S_1'$ is interpretable in $S_2'$ by the same rules, 
then $S_2 \equiv S_2'\Rightarrow S_1\equiv S_1'$ and 
$S_2 \prec S_2'\Rightarrow S_1\prec S_1'$. 
\end{Theorem}
Since $A=[0,u]$, $a\oplus b=(a+b)\wedge u$, $\neg 
a=u-a$, it follows that the $L_{MV}$-structure $A$ is interpretable in the $L_{lou}$-structure $(G_A,u_A)$ 
and in the $L_{loZu}$-structure $(G_A,\ZZ u_A)$. 
Consequently, if 
$A$, $A'$ are MV-algebras such that $(G_A,u_A)\equiv (G_{A'},u_{A'})$ (resp.\ 
$(G_A,\ZZ u_A)\equiv (G_{A'},\ZZ u_{A'})$), then $A\equiv A'$. The same holds with $\prec$ instead 
of $\equiv$. 
\section{Partially cyclically ordered groups.}\label{section3} 
\indent  Recall that all the groups are assumed to be abelian groups. 
\subsection{Basic properties.}
\begin{defis}\label{def31}
 We say that a group $C$ is 
{\it partially cyclically ordered} (in short a {\it p.c.o.\ group}) if it is equipped with 
a ternary relation $R$ which satisfies (1), (2), (3), (4) below. \\ 
(1) $R$ is strict i.e.\ for every $x$, $y$, $z$ in $C$: 
$R(x,y,z) \Rightarrow x \neq y \neq z \neq  x$.\\ 
(2) $R$ is cyclic i.e.\ for every $x$, $y$, $z$ in $C$:
$R(x,y,z) \Rightarrow R(y,z,x)$.\\ 
(3) For every $x$, $y$, $z$ in $C$ set $y\leq_x z$ if either 
$R(x,y,z)$ or $y=z$ or $y=x$. Then for every  $x$ in $C$, $\leq_x$ is a partial order 
relation on $C $. 
We set $y<_x z$ for $y\leq_x z$ and $y\neq z$. If 
$y$ and $z$ admit an infimum (resp.\ a supremum) in $(C,<_x)$, then it will be denoted 
by $y\wedge_xz$ (resp.\ $y\vee_xz$). \\
(4) $R$ is compatible, i.e.\ for every $x$, $y$, $z$, $v$ in $C$, $R(x,y,z) 
\Rightarrow R(x+v,y+v,z+v)$. \\ 
If for every $x\in C$ the order $\leq_x$ is a linear order, then we say that 
$C$ is a {\it cyclically ordered group} (in short a {\it c.o.\ group}). 
\end{defis}
\begin{defi}{\bf Notation.} The language $(0,+,-,R)$ of p.c.o.\ groups 
will be denoted by $L_c$. 
\end{defi}
\begin{defi} A {\it c-homomorphism} is a group homomorphism $f$ between two 
p.c.o.\ groups (or c.o.\ groups) such that for every $x,y,z$, if $R(x,y,z)$ holds and 
$f(x)\neq f(y)\neq f(z)\neq f(x)$, then $R(f(x),f(y),f(z))$ holds. 
\end{defi}
\begin{examps} Let $\UU$ be the multiplicative group of unimodular complex numbers. 
For $e^{i\theta_j}$ ($1\leq j\leq 3$) in $\UU$, such that $0\leq \theta_j<2\pi$, we let 
$R(e^{i\theta_1},e^{i\theta_2},e^{i\theta_3})$ if, and only if, either $\theta_1<\theta_2<
\theta_3$ or $\theta_2<\theta_3<\theta_1$ or $\theta_3<\theta_1<\theta_2$ (in 
other words, when one traverses the unit circle counterclockwise, sarting from 
$e^{i\theta_1}$ one finds first $e^{i\theta_2}$ then $e^{i\theta_3}$). Then $\UU$ is 
a c.o.\ group. One sees that the group $Tor(\UU)$ of torsion elements of $\UU$ 
(that is, the roots of $1$ in the field of complex numbers) is a c.o.\ subgroup. 
Now, let $(C_1,R_1)$ and $(C_2,R_2)$ be nontrivial c.o.\ groups and 
$C=C_1\times C_2$ 
be their cartesian product. For $(x_1,x_2)$, $(y_1,y_2)$ and $(z_1,z_2)$ in $C$, let 
$R((x_1,x_2),(y_1,y_2),(z_1,z_2))$ if, and only if, $R_1(x_1,y_1,z_1)$ and $R_2(x_2,y_2,z_2)$. 
Then $(C,R)$ is a p.c.o.\ group which is not a c.o.\ group. 
\end{examps}
\begin{examp} Any linearly ordered group is a c.o.\ group 
once equipped with the ternary relation: 
$R(x,y,z)$ iff $x<y<z$ or $y<z<x$ or $z<x<y$. In the same way, 
any partially ordered group is a p.c.o.\ group. 
\end{examp}
\indent  
In ordered sets, one often uses the notation $x_1<x_2<\cdots<x_n$. We define a similar 
notation for partial cyclic orders. 
\begin{defi}{\bf Notation.} Let $C$  be a p.c.o.\ group and 
$x_1,\dots,x_n$ in $C$, we will denote by $R(x_1,\dots,x_n)$ the formula: 
$R(x_1,x_2,x_3)\;\&\; R(x_1,x_3,x_4)\;\&\;\dots \;\&\;R(x_1,x_{n-1},x_n)$. 
\end{defi}
\indent  
In the unit circle $\UU$, $R(x_1,\dots,x_n)$ means that starting from $x_1$ one finds the 
elements $x_2,\dots,x_n$ in this order. 
\begin{Lemma} Let $C$  be a p.c.o.\ group and 
$x_1,\dots,x_n$ in $C$. \\ 
a) $R(x_1,\dots,x_n)\Leftrightarrow \forall (i,j,k) \in [1,n] \times [1,n]\times [1,n], \;
1\leq i<j<k\leq n\Rightarrow R(x_i,x_j,x_k)$. \\
b) $\forall y\in C,\; (R(x_1,\dots,x_n)\Leftrightarrow R(x_1+y,\dots,x_n+y))$. \\
c) $\forall i\in [1, n-1],\; (R(x_1,\dots,x_n)\Leftrightarrow 
R(x_{i+1},\dots,x_n,x_1,\dots,x_i))$. 
\end{Lemma}
\begin{proof} a) $\Leftarrow$ is straightforward. Assume that $R(x_1,\dots,x_n)$ 
holds 
and let $1\leq i<j<k\leq n$. Then $R(x_1,x_i,x_{i+1})$ and $R(x_1,x_{i+1},x_{i+2})$ hold.  
Therefore, 
since $<_{x_1}$ is transitive, $R(x_1,x_i,x_{i+2})$ holds, and so on. Hence $R(x_1,x_i,x_j)$ holds, and in 
the same way $R(x_1,x_j,x_k)$ holds. It follows that $R(x_j,x_1,x_i)$ and $R(x_j,x_k,x_1)$ hold. Hence 
$R(x_j,x_k,x_i)$ holds which implies that $R(x_i,x_j,x_k)$ holds. \\ 
\indent  b) For every $i$ in $[2,n-1]$, $R(x_1,x_i,x_{i+1})$ holds. Hence $R(x_1+y,x_i+y,
x_{i+1}+y)$ holds. Consequently, $R(x_1+y,\dots,x_n+y)$ also holds. \\ 
\indent  c) Assume that $R(x_1,\dots,x_n)$ holds. 
We have that $R(x_1,x_2,x_n)$ holds and by a) for every $i$ in 
$[3,n-1]$: $R(x_2, x_i,x_{i+1})$ holds. Therefore $R(x_2,\dots,x_n,x_1)$ holds. Now, c) follows by induction. 
\end{proof}
\indent  
If $C$ is a p.c.o.\ group, then by the definition $<_0$ is a partial order on 
the set $C$. Conversely, if a group is equipped with a partial order relation $<$, then we give a necessary 
and sufficient condition for $<$ being the order $<_0$ of some partial cyclic order. 
\begin{Proposition}\label{lem34} Let $C$ be a group. Then there 
exists a compatible partial cyclic order $R$ 
on $C$ if, and only if, there exists a partial order $<$ on the set 
$C \backslash \{ 0 \}$ such that for all $x$ and $y$ in 
$C \backslash \{ 0 \}$: $x<y \Rightarrow y-x < -x$. \\ 
If this holds, then we can set $R(x,y,z) \Leftrightarrow 0\neq y-x < z-x$, and $\leq$ is the 
restriction to $C \backslash \{ 0 \}$ of the relation $\leq_0$. 
\end{Proposition}
\begin{proof} Assume that $C$ is a p.c.o.\ 
group, and $x<_0 y$ in $C\backslash \{ 0\}$. Then $R(0,x,y)$ holds. Hence by compatibility: 
$R(-x,0,y-x)$ holds so $R(0,y-x,-x)$ holds i.e.\ $y-x<_0 -x$. \\ 
\indent  Assume that $<$ is a strict partial order on $C \backslash \{ 0 \}$ 
such that for all $x$ and $y$ in $C \backslash \{ 0 \}$ we have that  
$x<y \Leftrightarrow y-x<-x$. For all $x, \; y, \; z$ in $C$ set 
$R(x,y,z)$ if, and only if, $0\neq y-x < z-x$. \\ 
\indent  $R(x,y,z)$ implies $y-x \neq 0$, $z-x \neq 0$ and 
$y-x \neq z-x$ hence $x \neq y$, $x \neq z$ and $y \neq z$. \\ 
\indent  $R(x,y,z)$ implies $y-x<z-x$, hence $z-y=(z-x)-(y-x)<
-(y-x)=x-y$. Therefore: $R(y,z,x)$. \\ 
\indent  Let $v\in C$ and assume that $R(x,y,z)$ and $R(x,z,v)$ hold. Then $y-x<z-x$ and 
$z-x<v-x$, hence: $y-x<v-x$ i.e.\ $R(x,y,v)$, so $\leq_x$ is 
transitive. Now, if $R(x,y,z)$ holds, then $y-x<z-x$, hence $z-x \not< 
y-x$, i.e.\ $\neg R(x,z,y)$. It follows that  $\leq_x$ 
is a partial order. \\ 
\indent  Assume that $R(x,y,z)$ holds, then 
$0\neq y-x<z-x$ hence $0\neq (y+v)-(x+v)<(z+v)-(x+v)$ therefore 
$R(x+v,y+v,z+v)$. \end{proof}
\indent  
The relation $<_0$ can makes easier the construction of 
p.c.o.\ groups. For example, let $C=\ZZ/ 6 \ZZ = 
\{ 0, 1,2,3,-2,-1 \}$, and set $1 <_0 2$, $1 <_0 -1$, $-2 <_0 2$, 
and $-2 <_0 -1$. One can check that in this case, $<_0$ cannot be extended to a total 
order. \\ 
\indent  We know that if $<$ is a compatible partial order on a group, then $x<y\Rightarrow -y<-x$. The partial 
order $<_0$ on a p.c.o.\ group satisfies a weaker property. 
\begin{Lemma} In a p.c.o.\ group 
$C$, we have that $\forall x \in C\backslash \{0\}$, $\forall y \in C\backslash \{0\}$, $x <_0 y 
\Rightarrow \neg (-x <_0 -y).$ 
\end{Lemma} 
\begin{proof} Assume that $x<_0y$. Then: $y-x<_0-x$, and $-y = -x -(y-x) 
<_0 -(y-x)=x-y$. Now, assume that $-x<_0-y$. Then: $-y+x<_0x$, and: 
$y = x -(-y+x) <_0 -(-y+x)=y-x$. By transitivity: 
$-y <_0 x-y <_0 x <_0 y <_0 y-x <_0 -x <_0 -y$ implies: 
$-y <_0 -y$, a contradiction. \end{proof}
\indent  
Note that in the cyclically ordered case we deduce: $x<_0y\Rightarrow -y<_0-x$, since $<_0$ is 
a total order. 
\subsection{The subset of non-isolated elements.}
We define the set of non-isolated elements for the order $<_0$, and we prove that 
$<_0$ can be compatible in some cases. 
\begin{defi}\label{def316} Let $C$ be a p.c.o.\ group, 
we denote by $A(C)$ the set whose elements are $0$ and the $x\in 
C\backslash \{0\}$ such that there exists $y \in C\backslash \{0\}$ satisfying 
$x<_0 y$ or $y<_0 x$. The elements of $A(C)$ are called the {\it non-isolated} 
elements. 
\end{defi} 
\begin{Remarks}\label{rks36} 
1) Let $C$ be a p.c.o.\ group and $x$, $y$ in $A(C)\backslash \{0\}$.   
By Proposition 
\ref{lem34} since $x<_0 y\Rightarrow y-x<_0 -x$, if $x<_0 y$ in $A(C)$ and 
$z=y-x$, then $z\in A(C)\backslash \{0\}$. So there exists $z$ in $A(C)$ such that 
$y=x+z$. We see that this is similar to condition (i) in Remark \ref{rk26}. \\ 
2) If $-x<_0 y$, then $x+y<_0 x$.\\
3) 
Assume that $C$ and $C'$ are p.c.o.\ groups. It follows from Proposition 
\ref{lem34} that they are isomorphic if, and only if, there is a group isomorphism $\varphi$ 
from $C$ onto $C'$ such that for every $x$, $y$ in $C$: 
$$x\in A(C) \Leftrightarrow 
\varphi(x)\in A(C')\mbox{ }\&\mbox{ }x<_0y\Leftrightarrow \varphi(x)<_0\varphi(y).$$ 
\end{Remarks}
\begin{Proposition}\label{n38} (case of compatibility of $+$ and $<_0$). 
Let $C$ be a p.c.o.\ group such that 
for every $x$, $y$ in $A(C)\backslash \{0\}$ 
we have that $x<_0 y\Leftrightarrow -y<_0-x$ and 
$x+y\in A(C)\Leftrightarrow x\leq_0-y$ or $-y\leq_0 x$. Then for all $x$, $y$, $z$ in $A(C)
\backslash \{0\}$: 
$$(x<_0y\; \&\; 0<_0x+z<_0y+z) \Leftrightarrow (x<_0y<_0-z\mbox{ or } -z<_0x<_0y)$$ 
($\Leftarrow$ holds in every p.c.o.\ group). 
\end{Proposition}
\begin{proof}
In any p.c.o.\ group we have the following implications.  
$$\begin{array}{rcl}
0<_0x<_0y<_0-z &\Rightarrow &R(0,x,y,-z)\\
&\Rightarrow&R(0,x,y)\;\&\; R(x,y,-z)\\
&\Rightarrow &R(0,x,y)\;\&\; R(x+z,y+z,0)\\
&\Rightarrow &x<_0y\; \&\; 0<_0x+z<_0y+z.
\end{array}$$
In the same way:  
$$\begin{array}{rcl}0<_0-z<_0x<_0y &\Rightarrow &R(0,-z,x,y)\\
&\Rightarrow& R(0,x,y)\;\&\; R(-z,x,y)\\
&\Rightarrow &R(0,x,y)\;\&\; R(0,x+z,y+z)\\
&\Rightarrow &x<_0y\; \&\; x+z<_0y+z.\end{array}$$
Assume that $x<_0y$ and $0<_0x+z<_0y+z$ hold. It follows that $x+z$ and $y+z$ 
belong to $A(C)$ and by hypothesis, we have that either 
$x<_0-z$ or $-z<_0 x$. In the same way: either $y<_0-z$ or $-z<_0 y$. If 
$-z<_0 x$, then $-z<_0x<_0y$. If $y<_0-z$, then $x<_0y<_0-z$. It follows that $-z<_0x\;\&\;
y<_0-z$ does not hold. Now, $R(0,x+z,y+z)$ holds. Hence $R(-z,x,y)$ holds, so 
$x<_0-z<_0y$ does not hold. 
\end{proof}
\subsection{Wound-round p.c.o.\ groups.}
In the field $\CC$ of complex numbers, the multiplicative group $\UU$ of unimodular complex 
numbers is the image of the additive group $\RR$ of real numbers under the epimorphism 
$\theta\mapsto e^{i\theta}$. It follows that $\UU$ is isomorphic to the quotient group 
$\RR/2\pi\ZZ$. Then one can define the cyclic order on $\RR/2\pi\ZZ $ by: 
$R(x_1+2\pi\ZZ ,x_2+ 2\pi\ZZ ,x_3+ 2\pi\ZZ )$ if, and only if, there exists $x_j'$ in 
$[0,2\pi[$ such that $x_j-x_j'\in 2\pi\ZZ $ ($1\leq j\leq 3$) and $x_{\sigma(1)}'<
x_{\sigma(2)}'<x_{\sigma(3)}'$ for some $\sigma$ in 
the alternating group $A_3$ of degree $3$ (in other words, 
 $x_1'<x_2'<x_3'$ or $x_2'<x_3'<x_1'$ or $x_3'<x_1'<x_2'$). \\
\indent  More generally, if $(L,u)$ is a unital linearly ordered group, 
then the quotient group $L/\ZZ u$ can be cyclically ordered by setting 
$R(x_1+\ZZ u,x_2+ \ZZ u,x_3+ \ZZ u)$ if, and only if, there exists $x_j'$ in 
$[0,u[$ such that $x_j-x_j'\in \ZZ u$ ($1\leq j\leq 3$) and $x_{\sigma(1)}'<
x_{\sigma(2)}'<x_{\sigma(3)}'$ for some $\sigma$ in 
the alternating group $A_3$ of degree $3$ (\cite[p.\ 63]{Fu 63}). We say that 
$L/\ZZ u$ is the {\it wound-round} of $L$. Now, every c.o.\ group can be obtained 
in this way as shows the following theorem. 
\begin{Theorem}\label{thrieg} (Rieger, \cite{Fu 63}).  
Every c.o.\ group is the wound-round of a unique (up to isomorphism) unital 
linearly ordered group $(uw(C),u_C)$.
\end{Theorem}
\begin{defi} Let $C$ be a c.o.\ group. Then $uw(C)$ is called the 
{\it unwound} of $C$. \end{defi}
\begin{Corollary}\label{functco} 
The wound-round mapping defines a full and faithfull functor from the category of unital linearly ordered groups, 
together with unital increasing group 
homomorphisms, to the category of c.o.\ groups, together with c-homomorphisms. The unwound 
mapping defines a full and faithfull functor from the category of c.o.\ groups, together with c-homomorphisms, 
to the category of unital linearly ordered groups, 
together with unital increasing group homomorphisms. The composites 
of these two functors are equivalent to the identities of respective categories. 
\end{Corollary}
\indent  Now, we generalize this winding construction to partially ordered groups. 
\begin{Proposition}\label{prop37} 
Let $(G,<)$ be a partially ordered group, $0<u \in G$, $C$ be the quotient group 
$C=G/\ZZ u$ and $\rho$ be the canonical mapping from $G$ onto $C$.  \\ 
(1) For every $x$ and $y$ in $G$, there exists at most one $n \in \ZZ$ 
such that $x<y+nu<x+u$. \\ 
(2) For every $x_1$, $x_2$, $x_3$ in $G$, set 
$R(\rho(x_1),\rho(x_2),\rho(x_3))$ if, and only if, there 
exist $n_2$ and $n_3$ in $\ZZ$ such that $x_1 < x_2 + n_2 u < 
x_3+n_3u < x_1+u$. Then $(C,R)$ is a 
p.c.o.\ group. 
\end{Proposition}
\begin{proof}
(1) Assume that $n$ and $n'$ are integers such that $x<y+nu<x+u$ 
and $x<y+n'u<x+u$. Then $-x-u<-y-n'u<-x$. Therefore, by addition, 
$-u<(n-n')u<u$, hence $n-n'=0$. \\ 
\indent  (2) Assume that $x_1 < x_2 + n_2 u < x_3+n_3u < x_1+u$, 
let $x_1'$, $x_2'$, $x_3'$ in $G$ such that $\rho(x_i')=\rho(x_i)$ ($i\in \{1,2,3\}$), 
and let $n_1'$, $n_2'$, $n_3'$ be the integers such that $x_i'=x_i+n_iu$ ($i\in \{1,2,3\}$). 
Then $x_1'-n_1'u<x_2'-n_2'u+n_2u<x_3'-n_3'u+n_3u<x_1'-n_1'u$. Hence 
$x_1'<x_2'+(n_2+n_1'-n_2')u<x_3'+(n_3+n_1'-n_3')u<x_1'+u$. So $R$ is indeed 
a ternary relation on $C$. \\ 
\indent  We set $\rho(0)<_0\rho(x) <_0 \rho(y) \Leftrightarrow 
R(\rho(0), \rho(x) , \rho(y))$, and we prove that $<_0$ is a strict partial 
order relation such that $\rho(0)<_0\rho(x) <_0 \rho(y) \Leftrightarrow 
\rho(y)-\rho(x)<_0 -\rho(x)$. \\ 
\indent  By the definition, $\rho(0)<_0\rho(x) <_0 \rho(y)$ iff there exist 
$m$ and $n$ in $\ZZ$ such that $0<x+mu<y+nu<u$. It follows from (1)  
that $<_0$ is anti-reflexive and anti-symmetric. 
The transitivity is trivial. From $x+mu<y+nu<u$, it follows that 
$0<y-x+(n-m)u<-x+(1-m)u$. Now, we have that $-u<-y-nu<-x-mu<0$, 
hence $0<-y+(1-n)u<-x+(1-m)u<u$. This completes the inequality: 
$0<y-x+(n-m)u<-x+(1-m)u<u$, and consequently 
$\rho(y)-\rho(x)<_0 -\rho(x)$. Now, by Proposition \ref{lem34}, $G/\ZZ u$ is a 
p.c.o.\ group. \end{proof}
\begin{examp}
Let $G_1$ be the lexicographically ordered group 
$\RR \overrightarrow{\times} \RR$, and $G_2$ be the group 
$\RR \times \RR$ partially ordered in the following way: $(x,y) \leq (x',y') 
\Leftrightarrow x=x'$ and $y \leq y'$. We also define a partial cyclic order $R'$ on 
$\RR \times (\RR/\ZZ)$ by setting $R'((x_1,y_1),(x_2,y_2),(x_3,y_3))$ if, and only if, 
$x_1=x_2=x_3$ and $R(y_1,y_2,y_3)$, where $R$ is the cyclic order of $\RR/\ZZ$ 
defined in Proposition \ref{prop37}. In $G_1$ and $G_2$, let $u=(0,1)$, then 
$G_1/\ZZ u \simeq G_2/\ZZ u \simeq \RR \times (\RR/\ZZ)$ in the language of p.c.o.\ groups. \end{examp}
\begin{Remark}\label{rk38} In the proof of Proposition \ref{prop37}, 
we showed that the p.c.o.\ group $C=G/\ZZ u$ 
enjoys for all $x$, $y$: $0<_0x<_0y \Leftrightarrow 0<_0-y <_0 -x$. In particular, 
$x\in A(C) \Leftrightarrow -x\in A(C)$. 
\end{Remark} 
\begin{Remark} If $G$ is a partially ordered group and $0<u\in G$, then one can easily check that 
the subset $H:=\{x\in G\mid \exists (m,n)\in \ZZ\times\ZZ,\; mu\leq x\leq nu\}$ is a subgroup of 
$G$, and $u$ is a strong unit of $H$. Now, if $\rho(x)$ is a non isolated element of 
$C=G/\ZZ u$, then there exists $n\in \ZZ$ such that $0<x+nu<u$. In particular, 
$-nu<x<(1-n)u$. Hence, $x\in H$. It follows that we can restrict ourselves to the subgroup 
$H/\ZZ u$, or assume that $u$ is a strong unit of $G$. 
\end{Remark} 
\begin{defi}\label{def325} Let $C$ be a p.c.o.\ group.  
We will say that $C$ is a {\it wound-round} if there exists a unital partially ordered 
group $(G,u)$ such that $C\simeq G/\ZZ u$, partially 
cyclically ordered as in Proposition \ref{prop37}. If $G$ is an $\ell$-group, then we say that 
$C$ is the {\it wound-round of a lattice}. 
If $(G,u)$ is uniquely defined 
(up to isomorphism), then it is called the {\it unwound} of $C$. 
\end{defi}
\begin{Proposition}\label{funct} 
The wound-round mapping defines a functor  $\Theta$ from the category of unital 
partially ordered groups, together with unital increasing group 
homomorphisms, to the category of p.c.o.\ groups, together with c-homomorphisms. 
\end{Proposition}
\begin{proof} We prove that if $f$ is a unital increasing homomorphism between the unital partially 
ordered groups $(G,u)$ and $(G',u')$, then we can define a c-homomorphism between 
$C:=G/\ZZ u$ and $C':=G'/\ZZ u'$. Let $\rho$ (resp.\ $\rho'$) be the canonical epimorphism 
from $G$ onto $C$ (resp.\ from $G'$ onto $C'$). Since $f(\ZZ u)=\ZZ u'$, we can define a group homomorphism 
$\bar{f}$ between $C$ and $C'$ by setting for every $x\in G$ $\bar{f}(\rho(x))=\rho'(f(x))$. 
Let $x<y<z$ in $G$ such that $\bar{f}(\rho(x))\neq \bar{f}(\rho(y))\neq \bar{f}(\rho(z))\neq \bar{f}(\rho(x))$. 
Since $f$ is increasing, we have that $f(x)\leq f(y)\leq f(z)$. Now, $f(x)\neq f(y)\neq f(z)$, 
so we have that $f(x)<f(y)<f(z)$. We deduce that if $R(\rho(x),\rho(y),\rho(z))$ holds and 
$\bar{f}(\rho(x))\neq \bar{f}(\rho(y))\neq \bar{f}(\rho(z))\neq \bar{f}(\rho(x))$,  then 
$R(\bar{f}(\rho(x)), \bar{f}(\rho(y)), \bar{f}(\rho(z))$ holds. Hence $f$ is a 
c-homomorphism. Now, one can check that if $f\circ g$ is the composite of two unital increasing 
homomorphisms, then $\bar{f}\circ\bar{g}=\overline{f\circ g}$. 
\end{proof}
\indent  We turn to the first-order theory of the wound-round p.c.o.\ groups. 
\begin{Lemma}\label{lem38} 
Let $(G,u)$ be a unital partially ordered group, $C$ be the quotient group 
$C=G/\ZZ u$, $\rho$ be the canonical mapping from $G$ onto $C$ and 
$G_u=\{x\in G\mid x\geq 0 \;\&\; x \not\geq u\}$. 
Then: \\
$\bullet$ the restriction of $\rho$ to the subset $G_u$ is a one-to-one mapping onto $C$, \\
$\bullet$ for every $x$, $y$, $z$ in $G_u$, $\rho(x)+\rho(y)=\rho(z)\Leftrightarrow 
x+y-z\in \ZZ u$ and $R(\rho(x),\rho(y),\rho(z))$ if, and only if, either 
$x<y<z$ or $y<z<x$ or $z<x<y$. 
\end{Lemma}
\begin{proof} Since $u$ is a strong unit, for every $x\in G$ there exist integers 
$m$ and $n$ such that $m u\leq x <nu$, and we can assume that $m$ is maximal. 
Then $x-mu\in G_u$, hence the restriction of $\rho$ to $G_u$ is onto. 
Let $x$ and $y$ in $G_u$ such that $\rho(x)=\rho(y)$, then $x-y\in \ZZ u$. Hence there 
exists an integer $m$ such that $y=x+mu$, and without loss of generality we can assume that 
$m\geq 0$. If $m\geq 1$, then $y\geq x+u\geq u$: a contradiction. Hence $m=0$, and $y=x$. 
So the restriction of $\rho$ to $G_u$ is one-to-one. The remainder of the proof is 
straightforward using properties of subsection \ref{subsec21}.  
\end{proof}
\begin{Theorem}\label{prop38} 
Let $(G,u)$ be a unital partially ordered group, $C$ be the quotient group 
$C=G/\ZZ u$, $\rho$ be the canonical mapping from $G$ onto $C$ and 
$G_u=\{x\in G\mid x\geq 0 \;\&\; x \not\geq u\}$. 
Then: \\
$\bullet$ the p.c.o.\ group $C$, in the language $L_c$, is interpretable in 
$(G,\ZZ u)$ in the language $L_{loZu}$, \\ 
$\bullet$ if $(G',u')$ is a unital partially ordered group, 
then $(G,\ZZ u)\equiv (G',\ZZ u') \Rightarrow G/\ZZ u\equiv G'/\ZZ u'$ 
(the same holds with $\prec$ instead of $\equiv$). 
\end{Theorem}
\begin{proof} Follows from Lemma \ref{lem38} and properties of subsection \ref{subsec21}.  
\end{proof}
\indent  Now we focus on the sets of non-isolated elements of wound-round p.c.o.\ groups and 
of wound-rounds of lattices.  
\begin{Proposition}\label{rkm16} Let $(G,u)$ be a unital partially ordered group, 
$N=\{x\in\; ]0,u[\; \mid \exists y\in\; ]0,u[\; 
x<y\mbox{ or }y<x\}$, and $C$ be the wound-round p.c.o.\ group $G/\ZZ u$. 
Then, $(N,<)$ and $(A(C)\backslash\{0\},\leq_0)$ are isomorphic ordered sets. 
\end{Proposition}
\begin{proof} It follows from (1) of Proposition \ref{prop37} that the restriction of $\rho$
to $[0,u[$ is one-to-one. In particular, its restriction to $N$ is one-to-one. 
If $x\in N$, then there exists $y\in N$ such that $0<x<y$ or $0<y<x$. It follows that 
$\rho(x)<_0\rho(y)$ or $\rho(y)<_0\rho(x)$. In particular, $\rho(x)\in A(C)$, and 
$\rho$ is an homomorphism of ordered sets from $N$ to $A(C)$. Now, let $x\in G$ such that 
$\rho(x)\in A(C)$ and $x\notin \ZZ u$. Then, there exist $y\in G$ and integers 
$n$, $n'$ such that $0<x+nu<y+n'u<u$ or $0<y+n'u<x+nu<u$. In any case $x+nu \in N$. 
Since $\rho(x+nu)=\rho(x)$, it follows that $\rho$ is an isomorphism of ordered sets 
between $(N,\leq)$ and $(A(C)\backslash \{0\})$. 
\end{proof}
\indent  
We saw that an	MV-algebra $A$ is isomorphic to the subset $[0,u_A]$ of its Chang $\ell$-group 
$(G_A,u_A)$. We get a similar result in the case of wound-round $\ell$-groups. 
\begin{Corollary}\label{propap38} 
Let $(G,u)$ be a unital $\ell$-group, 
$C=G/\ZZ u$ and $\rho$ be the canonical mapping from $G$ onto $C$. 
We assume that $A(C)\neq \{0\}$. We add an element $1$ to $A(C)$ and we set 
$x<_01$ for every $x\in A(C)$. Then the ordered sets $([0,u],\leq)$ and 
$(A(C)\cup\{1\})$ are isomorphic. 
\end{Corollary}
\begin{proof} By the definition of $A(C)$ and of $\leq_0$ in $C=G/\ZZ u$, 
if $A(C)\neq \{0\}$, then there is $x$, $y$ in $G$ such that $0<x<y<u$ or $0<y<x<u$. 
Hence, by Lemma \ref{lm21}, $N=]0,u[$. Therefore the result follows from 
Proposition \ref{rkm16}. 
\end{proof}
\indent  Note that if $C$ is the wound-round of a lattice, $C\simeq G/\ZZ u$, 
with $u>0$ a strong unit of $G$, then for every $x$, $y$ in 
$A(C)$, we have the following: \\
\indent  $\bullet$  $0<x\leq_0 y\Leftrightarrow 0>-y\leq_0 -x$, \\
\indent  $\bullet$ $x \wedge_0 y$ exists (the infimum of $x$ and $y$ in $(A(C),\leq_0)$), \\
\indent  $\bullet$ $0$ is the smallest element, \\
\indent  $\bullet$ $x\vee_0 y$ does not exist if, and only if, $(-x)\wedge_0 (-y)=0$, and if this holds, then 
for every $z\in A(C)$: $x<_0 z\Rightarrow y\not<_0 z$. \\
\indent  Furthemore, for every $x\in A(C)$ there exists a unique $g\in G$ such that $0\leq g<u$ 
and $\rho(g)=x$. Now, if $x\in C \backslash A(C)$, then, for every $y\in C$, 
$x\wedge_0 y=0$, and $x\vee_0 y$ does not exist. 
\begin{Remark}\label{rk321} 
Let $(G,u)$ be a unital $\ell$-group and $C=G/\ZZ u$. 
Assume that $A(C)$ is not trivial. 
By Corollary \ref{propap38} the ordered sets $[0,u[$ and $A(C)$ are 
isomorphic. Now, since $G$ is lattice-ordered, we know that it is generated by its positive 
elements. We saw in Lemma \ref{lm27} that every positive element of $G$ is a sum of 
elements of $[0,u]$. It follows that the subgroup generated by $A(C)$ is equal to $C$. 
\end{Remark} 
\indent  
In general, $A(C)$ is not a subgroup of $C$. Now, we show that $A(C)$ is partially closed under $+$.  
\begin{Proposition}\label{prop320} (sums of non-isolated elements). 
Let $C$ be a  wound-round p.c.o.\ group. 
For every $x$, $y$ in $A(C)$ 
we have that $y-x\in A(C)\Leftrightarrow x\leq_0 y \mbox{ or } y\leq_0 x$. It follows that 
$x+y\in A(C)\Leftrightarrow 
x\leq_0 -y \mbox{ or } y\leq_0- x$. 
\end{Proposition}
\begin{proof} Let $(G,u)$ be a unital partially ordered group 
such that $C\simeq G/\ZZ u$, partially cyclically ordered as in Proposition \ref{prop37}. 
If $x=y$ or $x=0$ or $y=0$, then the result is trivial. Let $x\neq y$ in $A(C)\backslash \{0\}$. 
If $x<_0 y$, then we have already seen that 
by Proposition \ref{lem34} we have that $y-x\in A(C)$. If $y<_0 x$, then $x-y\in A(C)$, and by 
Remark \ref{rk38} we have that $y-x\in A(C)$. \\ 
\indent  Now, assume that $y-x\in A(C)$, and let $g$, $h$ in $G$ such that $x=\rho(g)$ and 
$y=\rho(h)$. We know that we can assume that $0< g<u$ and $0< h<u$. Therefore: 
$-u<h-g<u$. Now, $\rho(h-g)=y-x\in A(C)$, hence there exists an unique integer $n$ such that 
$0\leq h-g-nu<u$. It follows that either $-u<h-g<0$ or $0\leq h-g<u$. If $-u<h-g<0$, then 
$h<g$, and since $0<h$ and $g<u$, we have that $y<_0 x$. If $0\leq h-g<u$, then 
$g<h$, and since $0<g$ and $h<u$, we have that $x<_0 y$. The other assertion follows easily.
\end{proof}
\subsection{$\ell$-c.o.\ groups.}
In \cite{Glu 83}, the lattice-cyclically-ordered groups are defined to be p.c.o.\ groups 
such that $\leq_0$ defines a structure of distributive lattice with first element. 
In the present paper we look at a larger class of groups. 
Indeed, we noticed after 
Definition \ref{def325} that in the case of wound-rounds of $\ell$-groups 
for any nonzero $x$ and $y$, $x\vee_0 y$ exists if, 
and only if, $(-x)\wedge_0 (-y)\neq 0$. This motivate the following definition. 
\begin{defis} An {\it $\ell$-c.o.\ group} is a p.c.o.\ group $C$ such that, for 
every $x$ and $y$ in $A(C)$, $x\wedge_0 y$ exists, and $0<_0x<_0y\Leftrightarrow 0<_0-y<_0-x$. \\
An {\it $\ell$c-homomorphism} is a c-homomorphism from a $\ell$-c.o. group $C$ to 
a $\ell$-c.o. group $C'$ such that for every $x$ and $y$ in $C$ we have that 
$f(x\wedge_0 y)=f(x)\wedge_0 f(y)$. 
\end{defis}
\indent  From the properties that we noticed after Corollary \ref{propap38}, it follows that 
the wound-round of an $\ell$-group is an $\ell$-c.o.\ group. However, the wound-round 
operation is not a functor from the category of $\ell$-groups to the category of 
$\ell$-c.o.\ groups. Indeed, let $G$ be the $\ell$-group $\RR\times\RR$, $u=(1,1)$, 
$G'=\RR$, $u'=1$, and $f:\; \RR\times\RR\rightarrow \RR$ be the natural projection onto 
the first component. Then, $f$ is an $\ell$-homomorphism, and $f(u)=u'$. Let $x=(\frac{1}{2},2)$ 
and $y=(\frac{1}{4},4)$. Both of $x$ and $y$ belong to $G_u\backslash [0,u[$. Hence, if 
$\rho$ (resp.\ $\rho'$) is the canonical epimorphism from $G$ onto $C=G/\ZZ u$ 
(resp.\ from $G'$ onto $C'=G'/\ZZ u'$), then $\rho(x)\notin A(C)$, 
$\rho(y)\notin A(C)$, so $\rho(x)\wedge_0 \rho(y)=\rho(0)$. Now, $f(x)\in [0,u'[$, $f(y)\in [0,u'[$ 
and $f(x)\wedge f(y)=\frac{1}{4}$. It follows that $\rho'(f(x))\wedge_0 \rho'(f(y))=\rho'(f(y))
\neq \rho'(0)$. Consequently, the c-homomorphism $\bar{f}:\; C\rightarrow C'$ induced by 
$f$ (see Proposition  \ref{funct}) is not an $\ell$-c-homomorphism. However, we have the following. 
\begin{Proposition} Let $(G,u)$ and $(G',u')$ be unital $\ell$-groups, 
$f$ be a one-to-one unital $\ell$-homomorphism from $(G,u)$ to $(G',u')$ and $\bar{f}$ be the 
c-homorphism defined in the proof of Proposition \ref{funct}. Then, 
$\bar{f}$ is an $\ell$-c-homomorphism from $C:=G/\ZZ u$ to $C':=G'/\ZZ u'$.
\end{Proposition}
\begin{proof} Let $\rho$ (resp.\ $\rho'$) be the canonical epimorphism 
from $G$ onto $C$ (resp.\ from $G'$ onto $C'$). We recall that 
$\bar{f}$ is defined by setting, for every $x\in G$, $\bar{f}(\rho(x))=\rho'(f(x))$. \\
\indent  First we let $x\in G_u$, and we prove that $\rho'(f(x))\in A(C')\Leftrightarrow \rho(x)\in A(C)$. \\
\indent  Recall that, since $x\in G_u$, we have that $\rho(x)\in A(C)\Leftrightarrow x\in [0,u[$. 
We saw in Lemmas \ref{lm27} and \ref{lem28} that there is a unique sequence $x_1,\dots,x_n$ such that 
$x=x_1+\cdots+x_n$, $x_1=x\wedge u$, $x_2=(x-x_1)\wedge u$, and so on. 
Since $f$ is a unital $\ell$-homomorphism, we have that  $f(x_1)=f(x)\wedge u'$, $f(x_2)=
(f(x)-f(x_1))\wedge u'$, and so on. Now, $f$ is one-to-one, so, for $1\leq i\leq n$, 
$f(x_i)= 0\Leftrightarrow x_i=0$. Hence $f(x_1), \dots, f(x_n)$ is the sequence associated with $f(x)$ 
as in Lemmas \ref{lm27} and \ref{lem28}. 
It follows that 
$$\rho(x)\in A(C)\Leftrightarrow x\in [0,u[\Leftrightarrow n=1\Leftrightarrow f(x)\in [0,u'[ 
\Leftrightarrow \rho'(f(x))\in A(C').$$
Let $x$, $y$ in $G_u$. If $x$ and $y$ belong to $[0,u[$, then $f(x)$ and $f(y)$ belong to 
$[0,u'[$ and $f(x\wedge y)=f(x)\wedge f(y)$. Therefore  
$$\bar{f}(\rho(x)\wedge_0 \rho(y))=\bar{f}(\rho(x\wedge y))=\rho'(f(x\wedge y))=
\rho'(f(x)\wedge f(y))=\rho'(f(x))\wedge_0 \rho'(f(y))=\bar{f}(\rho(x))\wedge_0\bar{f}(\rho(y)).$$
If $x\notin [0,u[$, then $\rho(x)\wedge_0 \rho(y)=0$. Now, we have proved that $f(x)\notin 
[0,u'[$, hence $\rho'(f(x))\wedge_0 \rho'(f(y))=0$. The case where $y\notin [0,u[$ is 
similar. 
\end{proof}
\subsection{Cyclically ordered groups elementarily equivalent to subgroups of $\UU$.}
In this subsection we list some results of \cite{LL 13}. 
\begin{defi}\label{def331} Let $C$ be a c.o.\ group. \\ 
1) $C$ is said to be {\it c-archimedean} if for every $x$ and $y$ in $C\backslash \{0\}$ 
there exists an 
integer $n >0$ such that $R(0,nx,y)$ does not hold (in other words, $y\leq_0 nx$, since 
$(C,\leq_0)$ is linearly ordered). \\
2) $C$ is said to be {\it discrete} if $(C,\leq_0)$ is a  discretely ordered set. \\
3) $C$ is said to be {\it c-regular} if for every integer $n\geq 2$ and every 
$0<_0x_1<_0\cdots<_0x_n$ in $C$ there exists $x\in C$ such that 
$x_1\leq_0 nx\leq_0 x_n$ and $x<_02x<_0\cdots <_0 (n-1)x<_0nx$. 
This is equivalent to saying that its unwound is a {\it regular} 
linearly ordered group, that is, for every $n\geq 2$ and every 
$0<x_1<\cdots<x_n$ in $uw(C)$ there exists $x\in uw(C)$ such that 
$x_1\leq nx\leq x_n$. \\
4) $C$ is said to be {\it pseudo-c-archimedean} if $C$ belongs to the elementary 
class generated by the c-archimedean c.o.\ groups. \\
5) $C$ is said to be {\it pseudofinite} if $C$ belongs to the elementary 
class generated by the finite c.o.\ groups. 
\end{defi}
\indent  
Note that $C$ is c-archimedean if, and only if, its unwound is archimedean, and 
$C$ is discrete if, and only if, its unwound is a discrete linearly-ordered group. 
\begin{notas} If $C$ is discrete, then 
the first positive element $\varepsilon_C$ of $C$ is definable, 
we can assume that it lies in the language. For a prime $p$, integers $n\in \NN^*$ and 
$k\in \{0,\dots,p^n-1\}$, we denote by $D_{p^n,k}$  the formula: 
$\exists x, \;R(0,x,2x,\dots, (p ^n-1)x)\& p^nx=k\varepsilon_C$. \end{notas}
\begin{defi} If $B$ is an abelian group and $p$ is a prime, then we define the $p$-th 
{\it prime invariant of Zakon} of $B$, denoted by $[p]B$, to be the maximum number 
of $p$-incongruent elements in $B$. In the infinite case, we set $[p]B=\infty$, 
without distinguishing between infinities of different cardinalities (see \cite{Za 61}). 
\end{defi}
\begin{Theorem}
1) A dense c.o.\ group is pseudo-c-archimedean 
if, and only if, it is c-regular. If this holds, then it is elementarily equivalent to some 
c-archimedean dense c.o.\ group.  \\ 
2) Any two dense c-regular c.o.\ groups are elementarily equivalent if, and only if, 
their torsion subgroups are isomorphic and they have the same family of prime invariants of 
Zakon. This in turn is equivalent to: their torsion subgroups are isomorphic and their unwounds have the 
same family of prime invariants of Zakon. 
\end{Theorem} 
\begin{Theorem} 1) 
Any two non-c-archimedean c-regular discrete c.o.\ groups are
elementarily equivalent if, and only if, they satisfy the same formulas $D_{p^n,k}$. \\
2) A c.o.\ group is pseudofinite if, and only if, it is 
discrete and c-regular. \\ 
3) Let $U$ be a non-principal ultrafilter on $\NN^*$, $C$ be the ultraproduct of the 
c.o.\ groups $\ZZ/n\ZZ$, $p$ be a prime, $n\in \NN^*$ and 
$k\in \{0,\dots,p^n-1\}$. Then $C$ satisfies the formula $D_{p^n,k}$ if, and only if, 
$p^n\NN^* -k\in U$. 
\end{Theorem}
\section{From MV-algebras to wound-rounds of lattices.}\label{section4}
\indent  The correspondence between MV-algebras and p.c.o.\ groups is defined as follows. 
Let $A$ be an MV-algebra and $(G_A,u_A)$ be its Chang $\ell$-group. 
We saw in Section \ref{section2} that $\Xi:\; A\mapsto (G_A,u_A)$ is a functor from the category 
of MV-algebras to the category of unital $\ell$-groups, 
where $(G_A,u_A)$ is the Chang $\ell$-group of $A$. Now, the wound-round 
functor  $\Theta$: $(G,u)\mapsto G/\ZZ u$ 
defined in Proposition \ref{funct} is a functor from the category of unital $\ell$-groups 
to the category of wound-rounds of lattices, together with the c-homomorphisms. 
So, this gives rise to a functor  $\Theta\Xi$ from the category of MV-algebras to the 
category of wound-rounds of lattices, together with the c-homomorphisms. \\ 
\indent  In this section,  we describe the correspondence between MV-algebras and wound-round of lattices.  
Then, we define the converse correspondence.
\begin{defi}{\bf Notation.} We denote by 
$C(A)$ the wound-round of lattice $G_A/\ZZ u_A$, and by $\rho$ the canonical epimorphism 
from $G_A$ onto $C(A)$, where, for $x\in G_A$, $\rho(x)\in C(A)$ is the class of $x$ modulo $\ZZ u_A$. 
Without loss of generality, we assume that $A \subset G_A$ and $1=u_A$, 
we denote by $\varphi$ the restriction of $\rho$ to  $[0,u_A[$. \end{defi} 
\subsection{Interpretability of $A$ in $C(A)$}
\indent  
Recall that $A(C(A))$ is the set of non-isolated elements of $C(A)$ (see Definition 
\ref{def316}). Assume that  $A\neq\{0,1\}$, $A\neq\{0,1,x\}$ and $A\neq\{0,1,x,\neg x\}$, 
for some $x$. Then by Corollary \ref{propap38}, $\varphi$ is an isomorphism of ordered sets 
between $([0,u_A[,\leq)$ and $(C(A),\leq_0)$. It follows that, for every $x$, $y$ in $[0,u_A[$, 
$\varphi(x\wedge y)=\varphi(x)\wedge_0 \varphi(y)$, and if $x\vee y<u_A$, then 
$\varphi(x\vee y)=\varphi(x)\vee_0 \varphi(y)$. \\ 
\indent  Note that if 
$A=\{0,1\}$, then $C(A)=\{0\}$. If $A=\{0,1,x\}$, then $C(A)\simeq\ZZ/2\ZZ$. If 
$A=\{0,1,x,\neg x\}$ is not an MV-chain, then $C(A)\simeq\ZZ/2\ZZ\times \ZZ/2\ZZ$, and 
in any case $A(C(A))=\{0\}$. If $A=\{0,1,x,\neg x\}$ is an MV-chain, then 
$C(A)\simeq\ZZ/4\ZZ$. 
In the following, we assume that $A\neq\{0,1\}$, $A\neq\{0,1,x\}$ and $A\neq\{0,1,x,\neg x\}$, 
for some $x$. 
\begin{Remark}
Let $x$ in $]0,u_A[$, since $\neg x=u_A-x$, we have that $\varphi(\neg x)=-\varphi(x)$. 
\end{Remark}
\begin{Proposition}\label{prop41} We add an element $\un$ to 
$A(C(A))$, and we set $\varphi(u_A)=\un$. For every $x\in [0,u_A[$ set 
$$\varphi(x)<_0\un\mbox{, }\neg \varphi(x) =
-\varphi(x)\mbox{ if }x\in ]0,1[\mbox{, }\neg \varphi(0)=\un\mbox{ and }\neg \un=\varphi(0).$$ 
Let $x$, $y$ in $[0,u_A[$, we have that $\rho(x\oplus y)=\varphi(x)\wedge_0 (\neg \varphi(y))
+\varphi(y)$. 
\end{Proposition}
\begin{proof} We know that $x\oplus y=(x+y)\wedge u_A$, hence 
$x\oplus y= x\wedge (u_A-y)+y$. Assume that $x\wedge (u_A-y)+y<1$ and $y\neq 0$ (the case 
$y=0$ being trivial). Hence 
$$\begin{array}{rcl}
\rho(x\oplus y)&=&\varphi(x\oplus y)\\
&=&\varphi(x\wedge (u_A-y)+y)\\
&=&\varphi(x\wedge (u_A-y))+\varphi(y)\\
&=&\varphi(x)\wedge_0 \varphi(u_A-y)+\varphi(y)\\
&=&\varphi(x)\wedge_0 (\neg\varphi(y))+\varphi(y).\end{array}$$ 
If $x\wedge (u_A-y)+y=u_A$ i.e.\ $x\wedge (u_A-y)=u_A-y$, then $u_A-y\leq x$ and $y\neq 0$. It 
follows that $0<_0\varphi(u_A-y)=-\varphi(y)\leq_0\varphi(x)$ and 
$$\begin{array}{rcl}
\varphi(x)\wedge_0(\neg\varphi(y))+\varphi(y)&=&
\varphi(x)\wedge_0(-\varphi(y))+\varphi(y)\\
&=&-\varphi(y)+\varphi(y)\\
&=&0\\
&=&\rho(u_A)\\
&=&
\rho(x\oplus y).\end{array}$$ 
\end{proof}
\begin{Corollary} The MV-algebra $A$, in the language $L_{MV}$, is interpretable in 
the $L_{lo}$-structure $A(C(A))$. 
In particular, if $A$ and $A'$ are MV-algebras such that $A(C(A))\equiv A(C(A'))$, then 
$A\equiv A'$. The same holds with $\prec$ instead of $\equiv$. 
\end{Corollary}
\begin{proof} For every $x$, $ y$ in $A(C(A))\cup\{ \un\}$ we set $\neg x= -x$ 
if $0\neq x\neq \un$, $\neg 0=\un$, $\neg \un=0$, and $x\oplus y=x\wedge_0\neg y+y$ if 
$x\wedge_0\neg y+y\neq 0$ or $x=y=0$, and we set $x\oplus y=\un$ otherwise. The remainder of the proof follows from Theorem \ref{th27}. 
\end{proof}
\subsection{MV-algebra associated with a p.c.o.\ group}
\begin{defi}{\bf Notation.}\label{notat45} Let $C$ be a p.c.o.\ group. We add an 
element $\un$ to $A(C)$ and we set, for every $x\in A(C)$, $x<_0\un$ and $\un+x=x+\un=x$. 
\end{defi}
\begin{defi}\label{def45} 
Let $C$ be a p.c.o.\ group. 
We will say that $A(C)$ {\it defines canonically an MV-algebra} if it satisfies the following.\\ 
1) For every $x$, $y$ in $A(C)\backslash \{0\}$ we have that $x<_0 y\Leftrightarrow -y<_0-x$. \\
2) $(A(C)\cup\{\un\},\leq_0)$ is a distributive lattice. \\
3) For every $x$, $y$ in $A(C)$, 
$x+y=x\wedge_0y+x\vee_0 y$. \\
4) For every $x$, $y$, $z$ in $A(C)\backslash \{0\}$, we have that 
$$x-y=(x\wedge_0(-z)+z)\wedge_0(-y)-(y\wedge_0(-z)+z)\wedge_0(-x).$$
We will denote by $\mathcal{AC}$ the class of p.c.o.\ groups $C$ such that $A(C)$ defines 
canonically an MV-algebra. 
\end{defi}
\indent  Note that by Conditions 1) and 2) the elements of $\mathcal{AC}$ are $\ell$-c.o.\ groups. \\
\indent  The aim of this subsection is to prove the following theorem. 
\begin{Theorem}\label{n44} 
Let $C\in \mathcal{AC}$. 
Set $\neg 0=\un$, $\neg \un=0$ and for $x\in A(C)\backslash\{0\}$ set $\neg x=-x$. For every $x$, $y$ 
in $A(C)\cup\{\un\}$ set \\
$x\oplus y= x\wedge_0 (\neg y)+y$ if $x\wedge_0 (\neg y)+y\neq 0$ or $x=y=0$, and \\
$x\oplus y=\un$ otherwise. \\
Then $A(C)\cup\{\un\}$ is an MV algebra with natural partial order $\leq_0$. 
\end{Theorem}
\begin{Corollary}\label{cor411} Let $C$ be a p.c.o.\ group.\\
$\bullet$ $A(C)$ defining canonically an MV-algebra is expressible by countably many first-order 
formulas of the language $L_c$.  \\
$\bullet$ If $A(C)$ defines canonically an MV-algebra, then the MV-algebra $A(C)\cup\{\un\}$ 
defined in Theorem \ref{n44} is interpretable in $C\cup\{\un\}$, where $\un$ is a new element. 
\end{Corollary}
\begin{Remark} Let $A$ be an MV-algebra such that there exist $x<y$ in $]0,1[$. 
Then the MV-algebra $A(C(A))\cup \{\un\}$ (together with the operations defined in 
Theorem \ref{n44}) is isomorphic to $A$. 
\end{Remark}
\begin{proof} Since $]0,1[$ contains $x<y$, we deduce from 
Corollary \ref{cor22} that $A(C(A))$ is nonempty. By Corollary \ref{propap38}, the 
canonical epimorphism $\rho$ from the Chang $\ell$-group $(G_A,u_A)$ of $A$ induces an 
isomorphism $\varphi$ 
between the lattices $[0,u_A]$ and $A(C(A))\cup \{\un\}$. Now, for $g$, $h$ in $]0,u_A[$, 
$\varphi(g)\wedge(-\varphi(h))+\varphi(h)=\varphi(g)\wedge\varphi(u_A-h)+\varphi(h)=
\varphi((g+h)\wedge u_A)$. Hence $\varphi(g)\wedge(\neg\varphi(h))+\varphi(h)=0$ if, 
and only if, either $g+h\geq u_A$ or $g+h=0$. Consequently, by Proposition \ref{prop41}, 
$\varphi$ is an isomorphism of MV-algebras. 
\end{proof}
\indent  The proof of Theorem \ref{n44} is based on the following lemmas. 
\begin{Lemma}\label{lemav410}. Let $C$ be an $\ell$-c.o.\ group and  $x$, $y$ 
in $A(C)\backslash\{0\}$. If $-y\neq x\wedge_0(-y)$, then 
$y<_0 x\wedge_0(-y)+y$. In particular, $x\wedge_0(-y)+y$ belongs to $A(C)$. 
\end{Lemma}
\begin{proof} Since $-y\neq x\wedge_0(-y)$, we have that $x\wedge_0(-y)<_0-y$. 
By hypothesis, this is equivalent to $y<_0-(x\wedge_0(-y))$. 
By Proposition \ref{lem34}, this in turn is equivalent to $-x\wedge_0(-y)-y<_0-y$. 
By hypothesis, this in turn is equivalent to $y<_0 x\wedge_0(-y)+y$. 
The last assertion follows easily. 
\end{proof}
\begin{Lemma}\label{n39} 
Let $C$ be a p.c.o.\ group such that 
for every $x$, $y$ in $A(C)\backslash \{0\}$ we have that $x<_0 y\Leftrightarrow -y<_0-x$. Let
$x$, $y$ in $A(C)$ such that the infimum $z=x\wedge_0y$ of $x$ and $y$ in 
$(A(C),\leq_0)$ exists. Then $x-z$ and $y-z$ belong to $A(C)$, the infimum 
$(x-z)\wedge_0(y-z)$ exists and is equal to $0$. 
\end{Lemma}
\begin{proof} If $x\leq_0 y$, then $z=x$, $y-z=y-x<_0-x$ (Proposition \ref{lem34}). Hence $y-z\in A(C)$, 
and $(y-z)\wedge_0(x-z)=x-z=0$. The same holds if $y\leq_0 x$. Now, assume that nor 
$x\leq_0 y$ nor $y\leq_0 x$. 
We have that $z<_0x$, hence $x-z<_0-z$, in particular, $x-z\in A(C)$. Let $t\in C$ such that 
$0<_0t<_0x-z$. Then we have: $R(0,t,x-z,-z)$. Hence $R(z,t+z,x,0)$ holds. Therefore 
$R(0,z,t+z,x)$ holds, i.e.\ $z<_0t+z<_0x$. In the same way, $0<_0t<_0y-z\Rightarrow z<_0t+z<_0y$. 
Hence, since $z=x\wedge_0y$, this yields a contradiction. Consequently, 
there is no $t\in A(C)\backslash \{0\}$ such that $t<_0x-z\;\&\; t<_0y-z$. 
\end{proof}
\begin{Remark}\label{n310} Let $C$ be an $\ell$-c.o.\ group. Then, the supremum 
$x\vee_0y$ exists if, and only if, $(-x)\wedge_0(-y)\neq 0$. If this holds, then 
$x\vee_0y=-((-x)\wedge_0(-y))$. Otherwise, there is no $z\in A(C)$ such that $x\leq_0z$ and 
$y\leq_0z$. If the supremum of $x$ and $y$ does not exist, then we will set $x\vee_0y=\un$. So 
$(A(C)\cup\{\un\},\leq_0)$ is a lattice with smallest element $0$ and greatest element $\un$. 
\end{Remark}
\begin{Lemma}\label{lmav412} Let $C\in \mathcal{AC}$. Then, for every $x$, $y$ in 
$A(C)\cup\{\un\}$: $x\oplus y=\un\Leftrightarrow (-y\leq_0x\; \&\; (x,y)\neq (0,0))$. 
In particular: $\neg x\oplus x=\un$. 
\end{Lemma}
\begin{proof}
We have that $x\wedge_0 \neg y+y=0$ if, and only if, $-y=x\wedge_0(-y)$. So $x\wedge_0 \neg y+y=0
\Leftrightarrow -y\leq_0x$. In particular, $x\oplus y=\un \Leftrightarrow (-y\leq_0 x$ and 
$(x,y)\neq(0,0))$. 
\end{proof}
{\it Proof of Theorem \ref{n44}.} Note that if $x\vee_0y$ does not exist in $A(C)$, then 
$x\vee_0y=\un$. By Lemma \ref{lemav410}, if $x$ and $y$ belong to $A(C)\backslash \{0\}$ 
and $x\oplus y\neq \un$, then 
$x\oplus y=x\wedge_0 \neg y+y\in A(C)$. 
Let $x$, $y$ in $A(C)\cup\{\un\}$. If $y=0$, then $x\oplus y=x\wedge_0 \un +0=x$.  
If $x=0$, then $x\oplus y=0+y=y$. If $y=0$, then $x\oplus y=x+0=x$. 
If $y=\un $, then $x\wedge_0 0+\un =0$, hence $x\oplus y=\un $. 
If $x=\un $, then $1\wedge \neg y+y=\neg y+y=0$. Hence $x\oplus y=\un $. It follows that in any case 
$x\oplus y\in A(C)\cup\{\un\}$.\\ 
\indent  We have to prove that $\oplus$ and $\neg$ satisfy the axioms of Definition \ref{def21}. \\[2mm]
\indent  MV4) Trivially, for every $x\in A(C)\cup\{1\}$: $\neg \neg x=x$. \\[2mm]
\indent  MV3) and MV5) have already been proved (i.e.\ $x\oplus 0=x$, $x\oplus \neg 0=\neg 0$).  \\[2mm] 
\indent  MV2) ($x\oplus y=y\oplus x$) The case where $x\in \{0,\un\}$ or $y\in \{0,\un\}$ 
follows from above calculations. Assume that 
$x$ and $y$ belong to $A(C)\backslash \{0\}$. By Lemma \ref{lmav412}, 
$x\oplus y=\un \Leftrightarrow -y\leq_0 x\Leftrightarrow -x\leq_0 y\Leftrightarrow y\oplus x =\un $. 
Otherwise,  
$x\oplus y-y\oplus x=x\wedge_0 (-y)+y-(y\wedge_0(-x)+x)=x\wedge_0(-y)-(y\wedge_0(-x))-
(x-y)=x\wedge_0(-y)+(-y)\vee_0 x -(x-y)=0$, by 3) of Definition \ref{def45}. \\[2mm] 
\indent  MV6) ($\neg(\neg x \oplus y)\oplus y=\neg(\neg y \oplus x)\oplus x$) Trivially, we can assume that 
$x\neq y$. 
Since $x\vee_0y=y\vee_0x$, it is sufficient to prove that for all $x$, $y$ in $A(C)\cup\{\un\}$ 
we have that $\neg(\neg x \oplus y)\oplus y=x\vee_0y$. 
If $y=0$, then $\neg(\neg x \oplus y)\oplus y=\neg(\neg x \oplus y)=
\neg(\neg x)=x=x\vee_0 y$. If $x=0$, then $\neg(\neg x \oplus y)\oplus y=
\neg(\un\oplus y)\oplus y=\neg \un\oplus y=0\oplus y=y=x\vee_0 y$. \\ 
\indent  If $y=\un$, then $\neg(\neg x \oplus y)\oplus y=\neg(\neg x \oplus y)\oplus \un=\un=
x\vee_0 y$. 
If $x=\un$ and $y\in A(C)\backslash \{0\}$, then $\neg(\neg x \oplus y)\oplus y=\neg y\oplus y 
=\un =x\vee_0 y$. \\
\indent  If $x<_0 y$, then, by Lemma \ref{lmav412}, $\neg x\oplus y=\un $, and 
$\neg(\neg x \oplus y)\oplus y=0\oplus y=y =x\vee_0 y$. \\
\indent  Otherwise, we have that $\neg x\oplus y=(-x)\wedge_0(-y)+y\neq 0$, and 
$$\neg(\neg x\oplus y)\oplus y=(-((-x)\wedge_0(-y)+y))\oplus y=(x\vee_0y-y)
\oplus y=(x\vee_0y-y)\wedge_0(-y)+y.$$ 
Since $y<_0x\vee_0y$, we have that $x\vee_0y-y<_0-y$ 
(by Proposition \ref{lem34}). 
Hence $(x\vee_0y-y)\wedge_0(-y)+y=x\vee_0 y-y+y=x\vee_0y$. \\[2mm] 
\indent  MV1) ($x\oplus (y\oplus z)=(x\oplus y)\oplus z$). This is trivial if $x$, $y$ or $z$ 
belongs to $\{0,\un \}$. We assume that $x$, $y$, $z$ belong to $A(C)\backslash \{0\}$. 
Assume that $x\oplus y=\un$. Then, $(x\oplus y)\oplus z=\un$. 
By Lemma \ref{lmav412}, we have that $-y\leq_0x$. If $y\oplus z=\un$, then 
$(x\oplus y)\oplus z=\un=x\oplus (y\oplus z)$. We assume that $y\oplus z\neq \un$. 
By Lemma \ref{lemav410}, $y\leq_0 y\oplus z$. Hence $-(y\oplus z)\leq_0-y\leq_0x$. 
Therefore, $x\wedge_0 (-(y\oplus z))+(y\oplus z)=-(y\oplus z)+(y\oplus z)=0$. It follows that 
$(x\oplus y)\oplus z=\un=x\oplus (y\oplus z)$.\\
\indent  Assume that $x\oplus y\neq \un \neq y\oplus z$, so we have that $-y\not\leq_0 x$ and 
$-y\not\leq_0z$. Therefore:
$$\begin{array}{rcl}
(x\oplus y)\oplus z-x\oplus (y\oplus z)&=&
(x\oplus y)\oplus z-(z\oplus y)\oplus x\\
&=&(x\wedge_0(-y)+y)\wedge_0(-z)+z-(z\wedge_0(-y)+y)\wedge_0(-x)-x\\
&=&(x\wedge_0(-y)+y)\wedge_0(-z)-(z\wedge_0(-y)+y)\wedge_0(-x)-(x-z).
\end{array}$$
Now, it follows from 4) of Definition \ref{def45} that 
$(x\oplus y)\oplus z-x\oplus (y\oplus z)=0$. 
\hfill $\qed$
\begin{Remark} 
Let $n_1$ and $n_2$ be integers, greater than $4$, $C_1$ be the c.o.\ 
group $\ZZ/ n_1\ZZ$ and $C_2$ be the c.o.\ group $\ZZ/n_2\ZZ $. $C_1$ 
and $C_2$ define MV-algebras. We can define a p.c.o.\ group 
$C_1\times C_2$ by setting $R((x_1,x_2),(y_1,y_2),(z_1,z_2))\Leftrightarrow 
R(x_1,y_1,z_1)\; \&\; R(x_2,y_2,z_2)$. Then 
$$A(C_1\times C_2)=C_1\times C_2\setminus 
\left[((\ZZ/ n_1\ZZ)\times\{0\}) \cup(\{0\}\times (\ZZ/n_2\ZZ))\cup \{(n_1-1,1),
(1,n_2-1) \}\right].$$ 
Now, $-(3,1)=(n_1-3,n_1-1)$ and $(1,3)$ belong to $A(C_1\times C_2)$, 
$(3,1)\not \leq_0 (1,3)$, $(1,3)\not\leq_0 (3,1)$, but $(1,3)-(3,1)=(n_1-2,2)\in 
A(C_1\times C_2)$. Hence the rule $x\in A(C),y\in A(C) \Rightarrow 
(x+y\in A(C) \Leftrightarrow x\leq_0 -y\mbox{ or } -y\leq_0 x)$ does not hold. Consequently 
$C_1\times C_2$ does not define canonically an MV-algebra. 
\end{Remark}
\indent  We can define another partial cyclic order on $C_1\times C_2$ by setting 
$(x_1,x_2)\leq _0 (y_1,y_2) \Leftrightarrow (x_1\leq_0 y_1\;\&\; x_2\leq_0 y_2)$. In 
this case $A(C_1\times C_2)=C_1\times C_2$, and we conclude in the same way that 
$C_1\times C_2$ does not define canonically an MV-algebra. \\ 
\indent  Now, by  Theorem \ref{prop321} the p.c.o.\ group 
$(\ZZ\times \ZZ)/\ZZ(n_1,n_2)$ defines canonically an MV-algebra. 
\subsection{Wound-rounds of lattices}
\begin{Theorem}\label{prop321} 
Let $C$ be the wound-round of a lattice. Then $C\in \mathcal{AC}$. 
Furthermore, if $C=G/\ZZ u$, then $A(C)\cup \{\un\}$ is isomorphic to the MV-algebra 
$\Gamma(G,u)$. 
\end{Theorem}
\begin{proof} We have to prove that $C$ satisfies conditions 1), 2), 3), 4) of Definition 
\ref{def45}. \\[2mm]
\indent  1) has been proved  in Remark \ref{rk38}. \\[2mm] 
\indent  2) Let $(G,u)$ be a unital $\ell$-group such that 
$C\simeq G/\ZZ u$ and $\rho$ be the natural mapping from $G$ onto $C$. 
By Lemma \ref{lem38} the restriction of $\rho$ is a one-to-one mapping from 
$G_u=\{g\in G\mid 0\leq g\;\&\; g\not\geq 0\}$ onto $C$. We saw in Proposition \ref{rkm16} 
that $A(C)$ can be identified with a subset of $[0,u[$. 
Let $g$, $h$ in $[0,u[$ such that $\rho(g)\in A(C)$ and $\rho(h)\in A(C)$. We have that $g<h\Leftrightarrow \rho(g)<_0\rho(h)$. It follows that $\rho(g\wedge h)\in A(C)$, 
$\rho(g\wedge h)=\rho(g)\wedge_0\rho(h)$, and if $g\vee h\neq u$, then 
$\rho(g\vee h)\in A(C)$, $\rho(g\vee h)=\rho(g)\vee_0\rho(h)$. By setting $\rho(u)=\un$, 
we have that $g\vee h=u\Leftrightarrow \rho(g)\vee_0\rho(h)=\un$, hence $A(C)\cup\{\un\}$ 
embeds into a sublattice of $[0,u]$, so it is a distributive lattice, with smallest element $0$ 
and greatest element $\un$. Note that by Corollary \ref{propap38}, if $A(C)\neq\{0\}$, then 
$A(C)\cup \{\un \}$ is isomorphic to the lattice $[0,u]$. \\[2mm] 
\indent  3) Let $x$, $y$ in $A(C)$, and $g$, $h$ be the elements of $[0,u[$ such that 
$\rho(g)=x$ and $\rho(h)=y$. Since $G$ is an $\ell$-group, we have that 
$g+h=g\wedge h+g\vee h$, with $0\leq g\wedge h<u$ and $0\leq g\vee h\leq u$. We saw in 
Corollary \ref{propap38} that $\rho$ induces an isomorphism of ordered sets between 
$[0,u[$ and $A(C)$. Hence 
$\rho(g\wedge h)=x\wedge_0 y$, and if $g\vee h<u$, then $\rho(g\vee h)=x\vee_0y$. 
If $g+h\in G_u$, then $g\wedge h+g\vee h\in G_u$. Hence $g\vee h<u$ and 
$x+y=\rho(g+h)=\rho(g\wedge h+g\vee h)=\rho(g\wedge h)+\rho(g\vee h)=x\wedge_0 y+ 
x\vee_0 y$. Assume that $g+h\notin G_u$, then $g+h-u\in G_u$. If $g\vee h<u$, then we have that 
$$\begin{array}{rcl}
x+y&=&\rho(g)+\rho(h)\\
&=&\rho(g+h-u)\\
&=&\rho(g\wedge h+g\vee h-u)\\
&=&\rho(g\wedge h)+\rho(g\vee h)\\
&=&x\wedge_0 y+x\vee_0 y.\end{array}$$ 
If $g\vee h=u$, then $x\vee_0y=\un$. Hence $x\wedge_0y+x\vee_0y=x\wedge_0y$ (see Notation \ref{notat45}).  
Then: 
$$\begin{array}{rcl}
x+y&=&\rho(g)+\rho(h)\\
&=&\rho(g+h-u)\\
&=&\rho(g\wedge h+g\vee h-u)\\
&=&\rho(g\wedge h)\\
&=&x\wedge_0 y.\end{array}$$ 
\indent  4) Let $x$, $y$, $z$ in $A(C)\backslash \{0\}$ and 
$g$, $h$, $k$ in $]0,u[$ such that $\rho(g)=x$, $\rho(h)=y$ and $\rho(k)=z$. 
$$(x\wedge_0(-z)+z)\wedge_0(-y)-(y\wedge_0(-z)+z)\wedge_0(-x)=$$
$$(\rho(g)\wedge_0\rho(u-k)+\rho(k))\wedge_0\rho(u-h)-(\rho(h)\wedge_0\rho(u-k)+
\rho(k))\wedge_0\rho(u-g)=$$
$$(\rho(g\wedge(u-k))+\rho(k))\wedge_0\rho(u-h)-(\rho((h)\wedge(u-k))+\rho(k))\wedge_0
\rho(u-g)=$$
$$\rho(g\wedge(u-k)+k)\wedge_0\rho(u-h)-\rho(h\wedge(u-k)+k)\wedge_0
\rho(u-g).$$
$g\wedge(u-k)<u-k$, hence $g\wedge(u-k)+k<u-k+k=u$, hence 
$\rho(g\wedge(u-k)+k)\in A(C)$, and in the same way 
$\rho(h\wedge(u-k)+k)\in A(C)$, so: 
$$(x\wedge_0(-z)+z)\wedge_0(-y)-(y\wedge_0(-z)+z)\wedge_0(-x)=$$
$$\rho((g\wedge(u-k)+k)\wedge (u-h))-\rho((h\wedge(u-k)+k)\wedge(u-g))=$$
$$\rho((g+k)\wedge u\wedge (u-h))-\rho((h+k)\wedge u \wedge(u-g))=$$
$$\rho((g+k)\wedge (u-h)-(h+k)\wedge(u-g))=
\rho((g+k+h)\wedge u-h-(h+k+g)\wedge u+g)=$$
$$\rho(g-h)=\rho(g)-\rho(h)=x-y.$$
The last assertion follows from Proposition \ref{prop41} and from the definition 
of the MV-algebra $A(C)\cup\{\un \}$ given in Theorem \ref{n44}. 
\end{proof}
\indent  
One can wonder if being the wound-round of a lattice can be characterized by first-order 
sentences. We will see that this holds if the p.c.o.\ group $C$ is 
generated by $A(C)$. 
This characterization relies on good sequences. 
\begin{defi}{\bf Notation.} Let $C\in \mathcal{AC}$. 
We denote by $\langle A(C)\rangle$ the subgroup 
of $C$ generated by $A(C)$.  
\end{defi}
\begin{Theorem}\label{thm416}  
Let $C$ be a p.c.o.\ group. \\ 
$\langle A(C)\rangle$ is the wound-round of a lattice if, and only if, it is isomorphic to the 
wound-round of the Chang $\ell$-group 
of the MV-algebra $A(C(A))\cup\{\un\}$. \\
$\langle A(C)\rangle$ being the wound-round of a lattice is expressible by countably many 
first-order formulas of the language $L_c$. 
\end{Theorem}
\begin{proof} Recall that the MV-algebra $A(C)\cup\{\un\}$ is first-order definable in $C$, 
by the rules defined in Theorem \ref{n44}. In particular, we can 
assume that $\oplus$ belongs to the language. Trivially, if $\langle A(C)\rangle$ 
is isomorphic to the wound-round of the Chang $\ell$-group 
of the MV-algebra $A(C(A))\cup\{\un\}$, then it is the 
wound-round of a lattice. \\
\indent  Assume that $\langle A(C)\rangle =G/\ZZ u$, where $(G,u)$ is a unital $\ell$-group. 
Let $A$ be the MV-algebra $A(C)\cup \{1\}$, $(G_A,u_A)$ be the 
Chang $\ell$-group of $A$ and $C'$ be the p.c.o.\ group $G_A/\ZZ u_A$. \\[2mm]
\indent  1 We know that $A\simeq \Gamma (G_A,u_A)$, and, by Theorem \ref{prop321}, $A\simeq \Gamma (G,u)$. 
By uniqueness of the Chang $\ell$-group, it follows that there is a unital $\ell$-isomorphism 
between $(G,u)$ and $(G_A,u_A)$. 
Hence the p.c.o.\ groups $\langle A(C)\rangle$ and $C'$ are isomorphic. \\[2mm]
\indent  2 By Remark \ref{rk212}, every element $x$ of the positive cone of $G_A$ is a sum of 
elements $x_1,\dots,x_n$ of $A$ satisfying the conditions of Lemmas \ref{lm27} and \ref{lem28}, 
where $x\leq nu$. Furthermore, if $x\ngeq u_A$, then the $x_i$'s are different from $u_A$. Let 
call the sequence $(x_1,\dots,x_n,0,\dots)$ the {\it good sequence} associated with $x$. By 
Lemma \ref{lem38}, the canonical epimorphism $\rho:\; G_A\rightarrow G_A/\ZZ u_A$ induces a 
one-to-one mapping between $G_{u_A}=\{x\in G_A\mid x\geq 0\;\&\; x\ngeq u_A\}$ and 
$\langle A(C)\rangle$. It follows that every element $x$ of $\langle A(C)\rangle$ can be represented 
by a unique good sequence of elements of $A(C)$. Furthermore, by Lemma \ref{lm27}, if 
$x$ is a sum of $n$ elements of $A(C)$, then the good sequence associated with $x$ contains at most $n$ elements 
different from $0$. So $C$ satisfies the following family of first-order formulas. For every $n\in \NN^*$,  
$$\forall (x_1,\dots,x_n)\in A(C)^n\; \exists (y_1,\dots,y_n)\in A(C)^n \bigwedge_{1\leq i<n}
(y_{i+1}\wedge_0 -y_i=0 \;\&\; y_i=0\Rightarrow y_{i+1}=0)$$ 
$$\;\&\; x_1+\cdots+x_n=y_1+\cdots+y_n \;\&\; 
(\forall (z_1,\dots,z_n)\in A(C)^n \bigwedge_{1\leq i<n}
(z_{i+1}\wedge_0 -z_i=0 \;\&\; z_i=0\Rightarrow z_{i+1}=0)$$
$$\;\&\; x_1+\cdots+x_n=z_1+\cdots+z_n)
\Rightarrow z_1=y_1,\dots, z_n=y_n$$
Every element $x$ of the positive cone of $G_A$ is equivalent modulo $\ZZ u_A$ to an element 
$x'$ of $G_{u_A}$, and the good sequence associated with $x'$ is obtained by dropping the $u_A$'s from 
the good sequence associated with $x$. Hence so is the good sequence associated with $\rho(x)$. Now, if 
$x$ and $y$ belong to $G_{u_A}$, then the good sequence associated with $z=x+y$ is obtained by the 
rules $z_i=x_i\oplus(x_{i-1}\odot y_1)\oplus \cdots \oplus (x_1\odot y_{i-1})\oplus y_i$. The good 
sequence associated with $\rho(x)$ is obtained by dropping the $u_A$'s from the good 
sequence associated with $z$. Consequently, $C$ satisfies the following family of first-order 
formulas. For every $n\in \NN^*$, 
$$\forall (x_1,\dots,x_n)\in A(C)^n\; \forall (y_1,\dots,y_n)\in A(C)^n\; 
\forall (z_1,\dots,z_{n})\in A(C)^{n} \bigwedge_{1\leq i<n}
(x_{i+1}\wedge_0 -x_i=0$$ 
$$\;\&\; x_i=0\Rightarrow x_{i+1}=0) \;\&\; \bigwedge_{1\leq i<n}
(y_{i+1}\wedge_0 -y_i=0 \;\&\; y_i=0\Rightarrow y_{i+1}=0) \;\&\;\bigwedge_{1\leq i<n}
(z_{i+1}\wedge_0 -z_i=0$$
$$\;\&\; z_i=0\Rightarrow z_{i+1}=0) \;\&\; x_1+\cdots+x_n+y_1+\cdots+y_n=z_1+\cdots+z_n
\Rightarrow \bigcup_{1\leq i_0<n}$$
$$\bigcup_{1\leq i<i_0}
(x_i\oplus (x_{i-1}\odot y_1)\oplus\cdots\oplus(x_1\odot y_{i-1})\oplus y_i=1)\;\&\; 
x_{i_0}\oplus(x_{i_0-1}\odot y_1)\oplus\cdots\oplus(x_1\odot y_{i_0-1})\oplus y_{i_0}\neq 1$$
$$\;\&\; 
\bigcup_{i_0\leq i\leq n} z_i=x_i\oplus(x_{i-1}\odot y_1)\oplus\cdots\oplus(x_1\odot y_{i-1})\oplus y_i
$$
Conversely, assume that $C\in \mathcal{AC}$ and that $C$ satisfies above families of formulas 
(recall that by Corollary \ref{cor411}, $C\in \mathcal{AC}$ 
is expressible by countably many first-order formulas). We prove that 
$\langle A(C)\rangle$ is isomorphic to the wound-round of $G_A/\ZZ u_A$, 
where $A$ is the MV-algebra $A(C)\cup\{1\}$. The group operation on 
$\langle A(C)\rangle$ is determined by $A(C)$ and by above formulas, which are also satisfied by the 
group $C'=G_A/\ZZ u_A$. It follows that the groups $\langle A(C)\rangle$ and $C'$ are isomorphic. 
Furthermore, the ordered sets $[0,u]$ and $(A(C)\cup \{\un\},\leq_0)$ are isomorphic. By 
Corollary \ref{propap38} they are isomorphic to $(A(C')\cup \{\un\},\leq_0)$. 
By Remarks \ref{rks36}, the p.c.o.\ groups $C'$ and $\langle A(C)\rangle$ 
are isomorphic. This proves that being the wound-round of a lattice is expressible by countably 
many first-order formulas. 
\end{proof}
\begin{Remark} If $\langle A(C)\rangle$ is the wound-round of a lattice and is 
not linearly ordered (when equipped with the partial order $\leq_0$), then it is infinite. 
\end{Remark}
\begin{proof} Let $(G,u)$ be a unital $\ell$-group 
such that $\langle A(C)\rangle\simeq G/\ZZ u$. If $\langle A(C)\rangle$ is not linearly ordered, 
then $G$ is not linearly ordered. Hence there exist $x>0$ and $y>0$ in $G$ such that 
$x\nleq y$ and $y\nleq x$. So $x\wedge y<x$ and $x\wedge y<y$. By taking $x-x\wedge y$
instead of $x$, and $y-x\wedge y$ instead of $y$, we can assume that  $x>0$ and $y>0$ and 
$x\wedge y=0$. By properties of $\ell$-groups, for every positive integer $n$ we have 
$nx \wedge y=0=x\wedge ny$ (this follows for example from 1.2.24 on p.\ 22 of \cite{BKW}). In particular, 
$x$ and $y$ are not strong units. It follows that, for every 
$n\in \NN^*$, $nx\not> u$, hence $x,2x,\dots,nx,\dots$ belong to different classes modulo 
$\ZZ u$, therefore $G/\ZZ u$ is infinite. 
\end{proof}
\section{Case of MV-chains.}\label{section5}
\indent  We know that every c.o.\ group $C$ is the wound-round of a unique 
(up to isomorphism) unital linearly ordered group $(uw(C),u_C)$ 
(see Theorem \ref{thrieg}). 
So there is a one-to-one correspondence between c.o.\ groups and unital 
linearly ordered groups. In fact, this correspondence is a
functorial one (see Corollary \ref{functco}). \\
\indent  We construct $(uw(C),u_C)$. The linearly ordered group 
$uw(C)$ is isomorphic to $\ZZ \times C$. The partial order $\leq$ is the lexicographic order of 
$(\ZZ,\leq)\times (C,\leq_0)$, $u_C=(1,0)$ and $(m,x)+(n,y)=(m+n,x+y)$ if $x=y=0$ or 
$\min_0(x,y)<_0 x+y$, and $(m,x)+(n,y)=(m+n+1,x+y)$ otherwise. \\
\indent  There is also a 
one-to-one correspondence between unital 
linearly ordered groups and MV-chains: 
a unital linearly ordered group $(G,u)$ is associated with the MV-chain 
$\Gamma(G,u)=[0,u]$ (see \cite[Lemma 6]{Ch}). 
Conversely, an MV-chain $A$ is associated with its Chang $\ell$-group $G_A$. 
Furthermore, this correspondence is a functorial one (see Section \ref{section2}). 
It follows a functorial one-to-one 
correspondence between MV-chains and c.o.\ groups. If $A$ is an MV-chain, 
then $C(A)=G_A/\ZZ u_A$ is a c.o.\ group. Note that if 
$C$ is a c.o.\ group with at least three elements, then $A(C)=C$. It follows that 
if $C$ contains at least three elements, then the unital linearly ordered groups 
$(uw(C), u_C)$ and $\Gamma(A(C)\cup \{\un\},\un)$ are isomorphic. \\ 
\indent  The following lemma shows that the construction of $\Gamma(A(C)\cup \{\un\},\un)$, 
in the linearly ordered case, is similar to the construction of the unwound of a c.o.\ group.
\begin{Lemma} (\cite[Lemmas 5 and 6]{Ch}) Let $A$ be an MV-chain. 
Then $G_A$ is isomorphic to $\ZZ\times (A\backslash\{1\})$ lexicographically 
ordered and with the rules: $(m,x)+(n,y)=(m+n,x\oplus y)$ if $x\oplus y< 1$ and 
$(m,x)+(n,y)=(m+n+1,x\odot y)$ otherwise. \end{Lemma}
\indent  
We will also need the following fact. 
\begin{Fact}\label{fact52}
If $\rho$ is the natural mapping from $uw(C)$ 
onto $C\simeq uw(C)/\ZZ u_C$, then for $g$, $h$ in $[0,u_C[$ we have that $\rho(g)<_0\rho(h)
\Leftrightarrow g<h$ and if $g\leq h\neq 0$, then $g<g+h<g+u_C$. So, $\rho(g+h)=\rho(g)+
\rho(h)$ if, and only if, $g+h<u_C$, which in turn is equivalent to: $\rho(g)<_0\rho(g)+\rho(h)$. 
Otherwise, we have that $\rho(g)+\rho(h)=\rho(g+h-u_C)$. 
\end{Fact}
\indent  
Now, we prove that this correspondence between MV-chains and c.o.\ 
groups also preserves elementary equivalence. 
\begin{Proposition}\label{prop53} 
Let $A$ be an MV-chain. \\ 
The c.o.\ group $C(A)$, in the language $L_c$, is interpretable in the 
$L_{MV}$-structure $A$. \\
The $L_{MV}$-structure $A$ is interpretable in the $L_c$ structure $C(A)\cup\{\un\}$. \\
If 
$A$ and $A'$ are MV-chains, then:\\ 
$A\equiv A'\Leftrightarrow C(A)\cup\{\un \}\equiv C(A')\cup\{\un \}\Leftrightarrow 
C(A)\equiv C(A')$, and \\
$A\prec A'\Leftrightarrow C(A)\cup\{\un \}\prec C(A')\cup\{\un \}\Leftrightarrow 
C(A)\prec C(A')$. 
\end{Proposition}
\begin{proof}
In the MV-chain $A$, the set $C(A)$ is interpreted by 
$A\backslash \{1\}$, the cyclic order is given by $R(x,y,z)\Leftrightarrow 
x<y<z$ or $y<z<x$ or $z<x<y$, the addition is given by $x+y=x\oplus y$ if $x\odot y=0$ 
and $x+y=x\odot y$ otherwise. Indeed, we saw in Section \ref{section2} that in $G_A$ 
$x\oplus y=(x+y)\wedge 1$, and $x\odot y=(x+y-1)\vee 0$. Since $G_A$ is linearly ordered, 
$x\oplus y=\min (x+y,1)$, and $x\odot y=\max (x+y,1)-1$. Let $z\in [0,1[$ such that 
$x+y-z\in \ZZ \cdot 1$, then $z=x+y$ if $x+y<1$ and $z=x+y-1$ otherwise. \\ 
\indent  In $C(A)\cup\{\un\}$, the set $A$ is intepreted by $C(A)\cup \{\un\}$, $\neg x$ is 
interpreted by $-x$ if 
$x\notin\{0,\un\}$, $\neg 0=\un$ and $\neg \un=0$. $\oplus$ is interpreted by 
$\un \oplus x=\un$ and for $x$, $y$ in $A(C)$ 
$x\oplus y=x+y$ if $x+y\neq 0$ and $\min_0(x,y)<_0 x+y$, $x\oplus y=\un$ if 
$x\neq 0\neq y$ and $x+y\leq_0 \min_0(x,y)$, and $0\oplus 0=0$. Indeed, we have seen 
in Fact \ref{fact52} that 
if $g$, $h$ are the elements of $[0,u_C[ \subset uw(C)$ such that $\rho(g)=x$ and $\rho(h)=y$, 
then $g+h<u_C \Leftrightarrow \min_0(x,y)<_0 x+y$. \\ 
\indent  It follows from Theorem \ref{th27} that 
$A\equiv A'\Rightarrow C(A)\equiv C(A')$ and $C(A)\cup\{\un \}\equiv C(A')\cup\{\un \} 
\Rightarrow A\equiv A'$. Now we see that $C(A)\equiv C(A') \Rightarrow 
C(A)\cup\{\un \}\equiv C(A')\cup\{\un \}$. The last proposition can be proved in the same way. 
\end{proof} 
\begin{defi}{\bf Notation.} We consider the language $L_{oMV}=(0,+,-,\leq,\oplus,
\neg)$. The $L_o$-structure $\ZZ$ will be seen as a $L_{oMV}$-structure where 
$x\oplus y =z \Leftrightarrow x=y=z=0$ and $\neg x=y\Leftrightarrow x=y=0$. If 
$A$ is an MV-chain, then it will be seen as a $L_{oMV}$-structure, where $x+y=z\Leftrightarrow 
x=y=z=0$ and $x-y=z\Leftrightarrow x=y=z=0$. 
\end{defi} 
\begin{Proposition}\label{prop55} 
Let $A$ be an MV-chain. \\ 
The $L_{MV}$-structure $A$ is interpretable in the $L_{loZu}$-structure $(G_A,\ZZ u_A)$ 
(resp.\ in the $L_{lou}$-structure $(G_A,u_A)$). \\
The $L_{loZu}$-structure $(G_A,\ZZ u_A)$ (resp.\ the $L_{lou}$-structure $(G_A,u_A)$) is interpretable in the $L_{oMV}$-structure $\ZZ \times A$. \\
If $A$ and $A'$ are MV-chains, then: \\ 
$(G_A,\ZZ u_A)\equiv (G_{A'},\ZZ u_{A'})\Leftrightarrow \ZZ\times A \equiv \ZZ\times A' 
\Leftrightarrow A\equiv A'$, and \\
$(G_A,u_A)\equiv (G_{A'},u_{A'})\Leftrightarrow \ZZ\times A \equiv \ZZ\times A' 
\Leftrightarrow A\equiv A'$. \\ 
The same holds with $\prec$ instead of $\equiv$. 
\end{Proposition}
\begin{proof}
In $(G_A,\ZZ u_A)$ (resp.\ in $(G_A,u_A)$), $1$ is the smallest positive 
element of $\ZZ u$ (resp.\ $1=u$), the set $A$ is interpreted by 
$\{x\in G_A\mid 0\leq x\leq u_A\}$, 
$x\oplus y =\min(x+y,u_A)$, $\neg x=u_A-x$. \\
\indent  In $\ZZ \times A$, $G_A$ is 
interpreted by $\ZZ \times (A\backslash \{u\})$, $\ZZ u$ is interpreted by $\ZZ \times \{0\}$ 
(resp.\ $u=(1,0)$). The order relation is the lexicographic order: $(m,x)\leq (n,y) 
\Leftrightarrow m<n$ or $(m=n$ and $x\leq y)$. The sum is defined by 
$(m,x)+(n,y)=(m+n,x\oplus y)$ if $x\oplus y <1$, and $(m,x)+(n,y)=(m+n+1,x\odot y)$ if 
$x\oplus y=1$. \\
\indent  It follows from Theorem \ref{th27}) that  
$(G_A,\ZZ u_A)\equiv (G_{A'},\ZZ u_{A'})\Rightarrow A\equiv A'$, $(G_A,u_A)\equiv (G_{A'},u_{A'})
\Rightarrow A\equiv A'$, $\ZZ\times A \equiv \ZZ\times A' \Rightarrow 
(G_A,\ZZ u_A)\equiv (G_{A'},\ZZ u_{A'})$  and $\ZZ\times A \equiv \ZZ\times A' \Rightarrow 
(G_A,u_A)\equiv (G_{A'},u_{A'})$ (the same holds with $\prec$). 
Now, we deduce from Theorem \ref{th26} that in the language $L_{oMV}$: 
$A\equiv A'\Rightarrow \ZZ\times A \equiv \ZZ\times A'$. Now, clearly, if 
$A\equiv A'$ in $L_{MV}$, then $A\equiv A'$ in $L_{oMV}$ (the same holds with $\prec)$. 
\end{proof} 
\indent  
Thanks to this transfert principle, we deduce from \cite{LL 13} similar results in the case of 
MV-chains. In particular we characterize pseudofinite and pseudo-hyperarchimedean MV-chains. 
\begin{defis}(\cite[Chapter 6]{CDM}) \\ 
1) In an ordered set, by an {\it atom}  we mean an element $x$ such that $x>0$
and whenever $y\leq x$ then either $y=0$ or $y=x$ (\cite[Definitions 6.4.2 and 6.7.1]{CDM}).\\ 
2) An $\ell$-group is {\it hyperarchimedean} if for every positive $x$ and $y$ there exists 
$n\in \NN^*$  such that $nx\wedge y=(n+1)x\wedge y$ (see \cite[Theorem 14.1.2]{BKW}). \\
3) An MV-algebra is {\it atomic} if for each 
$x\neq 0$ there is an atom $y$ with $y\leq x$. It is {\it atomless} if no element is an atom 
(\cite[Definition 6.7.1]{CDM}).\\
4) An element $x$ of an MV-algebra is {\it archimedean} if there exists $n\in \NN^*$
such that $\neg x\vee n.x=1$. This is equivalent to saying that there exists $n\in \NN$ 
such that $n.x=(n+1).x$ (\cite[Corollary 6.2.4]{CDM}). \\
5) An MV-algebra is {\it hyperarchimedean} if all its elements are archimedean (\cite[Definition 6.3.1]{CDM}). \\
6) An MV-algebra is {\it simple} if it embeds in $[0,1]_{\RR}$ (\cite[Theorem 3.5.1]{CDM}). 
\end{defis}
\indent  
Note that if an MV-chain $A$ is atomic, then it contains only one atom, and the underlying ordered set 
is discretely ordered. If it is atomless, then the underlying ordered set is densely ordered. \\
\indent  Saying that an MV-chain is hyperarchimedean is equivalent to saying that it is simple. \\
\indent  Recall the notations, for $x$ in an MV-algebra, $2.x=x\oplus x$, $x^2=x\odot x$ and so on.
\begin{defis} Let $A$ be an MV-chain. \\
1) We will say that $A$ is {\it regular} if for every integer $n\geq 2$ and every 
$0<x_1<\cdots<x_n$ in $A$ there exists $x\in A$ such that $x_1\leq n.x\leq x_n$, 
and $0<x<2.x<\cdots<(n-1).x<n.x$. \\
2) We will say that $A$ is {\it pseudo-simple} if $A$ belongs to the elementary 
class generated by the simple MV-chains. \\
3) We will say that $A$ is {\it pseudofinite} if $A$ belongs to the elementary 
class generated by the finite MV-chains. 
\end{defis}
\indent  
Let $A$ be an MV-chain, it is easy to see that $A$ is regular if, and only if, $C(A)$ 
is c-regular, and since the unwound of $C(A)$ is isomorphic to 
the Chang $\ell$-group $G_A$ of $A$, this is equivalent to saying that $G_A$ is regular 
(see Definition \ref{def331}). 
One can also see that $A$ is atomic if, and only if, $C(A)$ is discrete. Moreover, 
$A$ is simple if, and only if, $G_A$ is archimedean. Note that a linearly ordered group is 
hyperarchimedean if, and only if, it is archimedean. \\
\indent  In the MV-chain $A(C)$, the formula $R(0,x,2x,\dots, (p ^n-1)x)$ can be reformulated as 
$0<x<2.x<\cdots<(p^n-1).x$, (which is equivalent to $x\neq 0$ and 
$0=x^2=\cdots =x^{p^n-1}$, since 
$x^2=0\Leftrightarrow 2.x\leq 1$) 
hence we can define formulas $D_{p^n,k}$ in MV-chains. 
\begin{defi} If $A$ is an atomic and not simple MV-chain, then 
the atom $\varepsilon_A$ of $A$ (which is the smallest positive element) is definable, 
we can assume that it lies in the language. For a prime $p$, for $n\in \NN^*$ and $k\in \{0,\dots,p^n-1\}$, 
we denote by $D_{p^n,k}$  the formula: 
$\exists x, \;0<x<2.x<\cdots< (p ^n-1).x)\wedge p^n.x=k.\varepsilon_A$. 
\end{defi}
\indent  In the same way, the torsion subgroup has an analogue in MV-chains. 
\begin{defi} Let $x$ be an element of an MV-chain. 
We will say that $x$ is a {\it torsion element} if there exists $n\in \NN^*$, such that 
$n.x=1$ and $x^n=0$.  
\end{defi}
\indent  
Note that $x\in A\backslash \{1\}$ is a torsion element in the MV-chain $A$ if, and only if, 
it is a torsion element in the group $C(A)$. 
\begin{Lemma}\label{lmav59}  
Let $A$ and $(A_i)_{i\in \Ni^*}$ be MV-chains, $U$ be an ultrafilter 
on $\NN^*$ and $\Pi A_i$ be the ultraproduct of $(A_i)_{i\in \Ni^*}$. Then 
$A\equiv \Pi A_i\Leftrightarrow C(A)\equiv \Pi C(A_i)$. 
\end{Lemma}
\begin{proof} Let $\Phi$ be a $L_{MV}$-sentence and $\Phi_c$ be the corresponding 
$L_c$-sentence. Then: $A\models \Phi\Leftrightarrow C(A)\models \Phi_c$, and for every 
$i$ in $\NN^*$, $A_i\models \Phi\Leftrightarrow C(A_i)\models \Phi_c$. Hence 
$\{i\in \NN^* \mid A_i\models \Phi\} \in U \Leftrightarrow 
\{i\in \NN^* \mid C(A_i)\models \Phi_c\}\in U$. The equivalence follows. 
\end{proof}
\indent  So various theorems proved in \cite{LL 13}, can be expressed in terms of MV-chains. 
\begin{Theorem}\label{thm511}
An atomless MV-chain is pseudo-simple if, and only if, it is regular. 
\end{Theorem} 
\begin{Theorem}\label{th511a}
1) Any atomless regular MV-chain is elementarily equivalent to some simple atomless MV-chain.  \\ 
2) Any two atomless regular MV-chains are elementarily equivalent if, and only if, 
their subchain of torsion elements are isomorphic and their Chang $\ell$-groups have the same family 
of prime invariants of Zakon. 
\end{Theorem} 
\begin{Theorem}\label{th511} 
1) 
Any two infinite atomic regular MV-chains are
elementarily equivalent if, and only if, they satisfy the same formulas $D_{p^n,k}$. \\
2) An infinite MV-chain is pseudofinite if, and only if, it is atomic and regular. \\ 
3) Let $U$ be a non principal ultrafilter on $\NN^*$, $A$ be the ultraproduct of the 
MV-chains $[0,n]$, $p$ be a prime, $n\in \NN^*$ and 
$k\in \{0,\dots,p^n-1\}$. Then $A$ satisfies the formula $D_{p^n,k}$ if, and only if, 
$p^n\NN^* -k\in U$. 
\end{Theorem}
\section{Non-linearly ordered case.}\label{section6}
\indent  First we list some properties of abelian $\ell$-groups 
(see \cite{BKW}). The aim is to get a sufficient condition for an $\ell$-group 
being a cartesian product of finitely many linearly ordered groups. 
We let $G$ be an $\ell$-group. 
We know that, for every $x\in G$, there exists a unique 
pair $x_+$, $x_-$ of non-negative elements such that $x=x_++x_-$ and 
$x_+\wedge x_-=0$. We let $|x|:=x_++x_-$. \\
\indent  Two elements $x$, $y$ of $G$ are said to be {\it orthogonal} if $|x|\wedge |y|=0$. This 
is equivalent to: $x_+\wedge y_+=x_+\wedge y_-=x_-\wedge y_+=x_-\wedge y_-=0$. 
A subset $A$ of $G$ is said to be {\it orhogonal} if its elements are pairwise orthogonal.  
Every orthogonal subset is contained in a maximal orthogonal subset. \\
\indent  If $A\subset G$, then the {\it polar} of $A$ is the set $A^{\perp}:=\{y\in G\mid 
\forall x\in A,\; |x|\wedge |y|=0\}$; if $A=\{x\}$, then we let $x^{\perp}:=\{x\}^{\perp}$. The set 
$A^{\perp\perp}$ is called a {\it bipolar}. Every 
polar of $G$ is a convex $\ell$-subgroup of $G$. A polar $A^{\perp}$ 
is said to be {\it principal} if $A^{\perp}=x^{\perp \perp}$ for some $x\in G$
(see \cite[Chapter 3]{BKW}).\\ 
\indent  An element $x$ of $G_+$ is 
said to be {\it basic} if $x^{\perp\perp}$ is a linearly ordered group, which is equivalent to 
saying that the set $[0,x]$ is linearly ordered. If $x$ and $y$ are basic elements, then either 
$x\leq y$ or $y< x$ or $x\wedge y=0$. If $x\wedge y>0$, then $x^{\perp}=y^{\perp}$, 
hence $x^{\perp\perp}=y^{\perp\perp}$ (see \cite[pp.\ 133-135]{BKW}). \\ 
\indent  The group $G$ is said to be {\it projectable} if, for every $x\in G$, 
$G$ is the direct sum of $x^{\perp}$ and $x^{\perp\perp}$. Note that being projectable  
is a first-order property. Let $x$, $y$ in $G_+$. Then $y\in x^{\perp} \Leftrightarrow 
x\wedge y=0$, and $y\in x^{\perp\perp}\Leftrightarrow \forall z\; (x\wedge z=0\Rightarrow 
y\wedge z=0)$. Hence $y\in x^{\perp}$ and $y\in x^{\perp\perp}$ are first-order properties. 
\begin{Lemma} If $\{x_1,\dots, x_n\}$ is a maximal orthogonal set  of an $\ell$-group 
$G$ whose elements are basic elements, then: \\ 
$\bullet$ for every $x\in G$ there is some $i\in \{1,\dots,n\}$ such that 
$x_i^{\perp\perp}\subset x^{\perp\perp}$,\\ 
$\bullet$ $x>0$ is basic if, and only if, there exists $i\in \{1,\dots,n\}$ such that 
$x_i^{\perp\perp}=x^{\perp\perp}$, \\
$\bullet$ the minimal principal polars of $G$ are 
$x_1^{\perp\perp},\dots,x_n^{\perp\perp}$. 
\end{Lemma}
\begin{proof} Let $0< x \in G$. Since $\{x_1,\dots,x_n\}$ is maximal 
orthogonal, 
there is some $i\in \{1,\dots,n\}$ such that $x\wedge x_i>0$. Now, $[0,x_i]$ is linearly ordered, 
hence for every $y\in G_+$, we have that either $x_i\wedge x\leq x_i\wedge y$ 
or $x_i\wedge y\leq x_i\wedge x$. If $x_i\wedge y\leq x_i\wedge x$, then 
$x_i\wedge y\leq x_i\wedge (x\wedge y)$. Hence $y\in x^{\perp}\Rightarrow y\in 
x_i^{\perp}$. If $x_i\wedge x< x_i\wedge y$, then $x_i\wedge x\leq x_i\wedge (x\wedge y)$. 
Hence $y\wedge x>0$. It follows that $x^{\perp}\subset x_i^{\perp}$. Therefore 
$x_i^{\perp\perp}\subset x^{\perp\perp}$. \\ 
\indent  Let $x>0$ and $i\in \{1,\dots,n\}$ such that $x^{\perp}=x_i^{\perp}$. If $x$ is basic, 
then we have that $x_i\leq x$ or $x\leq x_i$. Assume that $x_i\leq x$. For every $y>0$ we have that 
$y\wedge x_i=y\wedge x_i\wedge x=y\wedge x \wedge x_i=\min(y\wedge x,x_i)$,  
since $[0,x]$ is linearly ordered. Therefore: $y\wedge x_i=0\Leftrightarrow y\wedge x=0$, 
hence $x^{\perp}=x_i^{\perp}$. The  case $x\leq x_i$ is similar. \\
\indent  The last assertion follows trivially. 
\end{proof} 
\indent  By \cite[Theor\`eme 6, Chapitre II]{Ja 51}, we know that an $\ell$-group $G$ 
is a direct sum of linearly ordered groups if, and only if, the following holds:\\
\indent  $\bullet$ for every $x^{\perp\perp}$ which is 
minimal, $G$ is the direct sum of $x^{\perp}$ and $x^{\perp\perp}$,\\
\indent  $\bullet$ every $x^{\perp\perp}$ contains 
some $y^{\perp\perp}$ which is minimal,\\
\indent  $\bullet$ there is at most a finite number of minimal $y^{\perp\perp}$.\\ 
\indent  It follows that $G$ is the direct sum of $n$ linearly ordered groups if, and only if, 
it contains a maximal orthogonal set $\{x_1,\dots,x_n\}$ whose elements are basic elements and, 
for every $i\in \{1,\dots,n\}$, 
$G$ is the direct sum of $x_i^{\perp}$ and $x_i^{\perp\perp}$.  
This is equivalent to saying that 
$G$ is projectable and contains a maximal orthogonal set of $n$ element which are basic. \\
\indent  Now, we have also the following result. 
\begin{Proposition}\label{prop62} 
Let $G$ be an $\ell$-group together with a distinguished strong 
unit $u$. Then $G$ is the product of $n$ linearly ordered groups if, and only if, 
$G_+$ contains a maximal orthogonal set $\{u_1,\dots,u_n\}$, whose elements are basic, 
such that $u=u_1+\cdots+u_n$. 
\end{Proposition}
\begin{proof} $\Rightarrow$ is straightforward. Assume that 
$G_+$ contains a maximal orthogonal set $\{u_1,\dots,u_n\}$, whose elements are basic, 
such that $u=u_1+\cdots+u_n$. 
We know that if $x$, $y$, $z$ in $G_+$ satisfy $x\wedge y=0$, then $x+y=x\vee y$ and 
$(x+z)\wedge y=z\wedge y$ (see, for example, \cite[Lemma 2.3.4]{Gl 99}). 
Let $x\in G_+$, and $p\in \NN^*$ such that $x\leq p u$. We have that $p u=
p u_1+\cdots +p u_n$, where the $p u_i$'s are pairwise orthogonal. Hence 
$x=x\wedge p u=x\wedge p u_1+\cdots+x\wedge p u_n 
\in u_1^{\perp\perp}+\cdots+u_n^{\perp\perp}$. Note that, since $\{u_1,\dots,u_n\}$ is 
a orthogonal, we have that $x=(x\wedge p u_1)\vee\cdots\vee(x\wedge p u_n)$. Assume that 
$x=x_1+\cdots+x_n=x_1\vee \cdots \vee x_n$ with $x_i\in u_i^{\perp\perp}$ 
($1\leq i\leq n$). Then $x_i=x\wedge pu_i$, which proves the uniqueness of the 
decomposition. It follows that 
$G$ is the direct sum of $u_1^{\perp\perp},\dots,u_n^{\perp\perp}$. 
\end{proof}
\indent  Now, we turn to MV-algebras. From \cite[Lemma 2.3.4]{Gl 99}, which we recalled 
in the proof of Proposition \ref{prop62}, one deduces by induction that for every orthogonal family 
$\{x_1,\dots,x_n\}$ in the positive cone of an $\ell$-group we have 
$x_1+\cdots+x_n=x_1\vee \cdots\vee x_n$. Now, in an MV-algebra if $\{x_1,\dots,x_n\}$ is 
an orthogonal family, then $x_1\oplus\cdots\oplus x_n=x_1\vee \cdots\vee x_n$. The following 
proposition is similar to Lemma 6.4.5 in \cite{CDM}. 
\begin{Proposition}.\label{prop64} Let $A$ be an MV-algebra. 
Then, the Chang $\ell$-group of $A$ is isomorphic to a product of $n$ 
linearly ordered groups if, and only if, 
there exist non zero elements $u_1,\dots,u_n$ of $A$ such that: \\
$\bullet$ $1=u_1\oplus \cdots \oplus u_n$,\\ 
$\bullet$ for all $i$, $j$ in $\{1,\dots, n\}$: $i\neq j\Rightarrow u_i\wedge u_j=0$,\\ 
$\bullet$ for all  $x$, $y$ in $A$, if $x\leq u_i $ and $y\leq u_i$, then $x\leq y$ or $y\leq x$. \\
If this holds, then the Chang $\ell$-group of $A$ is interpretable in 
$(\ZZ\times [0,u_1[)\times \cdots \times (\ZZ\times [0,u_n[)$, where 
$(p_1,x_1,\dots,p_n,x_n)\leq (q_1,y_1,\dots,q_n,y_n)$ if, and only if, for every $i\in 
\{1,\dots,n\}$, $p_i<q_i$ or ($p_i=q_i$ and $x_i\leq y_i$). 
The addition is defined componentwise, 
$(p_i,x_i)+(q_i,y_i)= (p_i+q_i,x_i\oplus y_i)$ if $x_i\oplus y_i < u_i$, and 
$(p_i,x_i)+(q_i,y_i)= (p_i+q_i+1,x_i\odot y_i)$ if $x_i\oplus y_i = u_i$. 
\end{Proposition} 
\begin{proof} The equivalence follows from Proposition \ref{prop62}. 
Now let $x\in G_{A+}$. We know that there exists a good sequence $(x_1,\dots,x_p)$ of 
elements of $[0,u]$ such that $x=x_1+\cdots+x_p$, where, for $1\leq k\leq p-1$, 
$(u-x_k)\wedge x_{k+1}=0$ (see Section \ref{section2}). For $j\in \{1,\dots,n\}$ let $x_j=x_{1,j}+
\cdots+x_{n,j}$, with 
$x_{i,j}\in u_i^{\perp\perp}$ ($1\leq i\leq n$). We have that $u-x_j=(u_1-x_{1,j})+\cdots+
(u_n-x_{n,j})$, hence $u_i-x_{i,j}>0\Rightarrow x_{i,{j+1}}=0$ i.e.\ $x_{i,j}\neq u_i\Rightarrow 
x_{i,{j+1}}=0$. Therefore we can write $x$ as $x=k_1u_1+x_1+\cdots+k_nu_n+x_n$, with 
$0\leq k_i\leq p$ and $x_i\in [0,u_i[$ ($1\leq i\leq n$). So, every element of $G_A$ 
can be writen in a unique way as $x=k_1u_1+x_1+\cdots+k_nu_n+x_n$, with 
$ k_i\in \ZZ$ and $x_i\in [0,u_i[$ ($1\leq i\leq n$). 
Let $x=k_1u_1+x_1+\cdots+k_nu_n+x_n$, and $y=l_1u_1+y_1+\cdots+l_nu_n+y_n$ in 
$G_A$. \\ 
\indent  Trivially, 
$x\leq y$ if, and only if, for every $i\in 
\{1,\dots,n\}$, $k_i<l_i$ or ($k_i=l_i$ and $x_i\leq y_i$). \\
\indent  Set $x+y=z=m_1u_1+z_1+\cdots+m_nu_n+z_n$. Since $k_iu_i+x_i+l_iu_i+y_i
\in u_i^{\perp\perp}$, we have that $m_iu_i+z_i=(k_i+l_i)u_i+x_i+y_i$, i.e.\ 
$x_i+y_i-z_i=(m_i-k_i-l_i)u_i$. If $x_i+y_i<u_i$, then $-u_i<x_i+y_i-z_i<u_i$, hence 
$x_i+y_i-z_i=0$ and $z_i=x_i+y_i=x_i\oplus y_i$ and $m_i=k_i+l_i$. Otherwise, in the 
same way we prove that $z_i=x_i+y_i-u_i$ and $m_i=k_i+l_i+1$. Now, $x_i\oplus y_i=
(x_i+y_i)\wedge u=(x_i+y_i)\wedge u_i=u_i$ and $x_i\odot y_i=u-[(2u_i-x_i-y_i)\wedge u]=
(x_i+y_i-u)\vee0=x_i+y_i-u_i$. 
\end{proof}
\begin{Remarks} Since in $G_A$ we have that $A=[0,u_A]$, saying that $G_A$ is 
isomorphic to a product of $n$ linearly ordered groups is equivalent to saying that $A$ is a 
isomorphic to a product of $n$ MV-chains. \\
It follows from Proposition \ref{prop64} that being isomorphic to a product of $n$ MV-chains 
is a first-order property. 
\end{Remarks}
\begin{Proposition} (\cite[Proposition 3.6.5]{CDM}). 
Let $A$ be a finite MV-algebra. Then $A$ is isomorphic to 
a product of finite MV-chains and its Chang $\ell$-group 
$G_A$ is isomorphic to some $\ZZ\times \cdots\times \ZZ$. 
\end{Proposition}
\indent  
Since an MV-algebra embeds in the positive cone of its Chang $\ell$-group, 
we have for every $a$: $|a|=a$. Hence we can define orthogonal elements  
and polars in the following way. 
\begin{defi} Let $A$ be an MV-algebra. Two elements $a$, $b$ are
{\it orthogonal} if $a\wedge b=0$. The polar of a subset $B$ of $A$ is $B^{\perp}=
\{a\in A\mid \forall b\in B\;a\wedge b=0\}$. The MV-algebra $A$ is said 
to be {\it projectable} if, for every $a$, $b$ in $A$, $b$ can be written in a unique way as 
$b=b_1\oplus b_2$, with $b_1\in a^{\perp}$ and $b_2\in a^{\perp\perp}$. 
\end{defi}
\indent  
One can also prove that an MV-algebra is projectable if, and only if, its Chang $\ell$-group is 
projectable. \\
\indent  We see that every finite MV-algebra is pojectable, and that every 
minimal principal polar is discrete. Consequently, every pseudofinite MV-algebra 
is projectable, and its minimal principal polars are discrete and regular. \\[2mm]
\indent  Turning to first-order theory, we consider the language 
$L_{MVn}=(0,\oplus,\neg, \underline{1_1}, \dots, 
\underline{1_n})$, with $n$ new constant symbols. Let $(A_1,1_1),$ $\dots,$
$(A_n,1_n)$, $(A_1',1_1'),$ $\dots,$ $(A_n',1_n')$ be MV-chains. For $i\in \{1,\dots,n\}$, we 
assume that $(A,1_i)$ is a $L_{MVn}$-structure, by setting, for $j\in \{1,\dots,n\}$, 
$x=\underline{1_j}$ if either $i=j$ and $x=1_i$, or $i\neq j$ and $x=0$. We define in the 
same way the $L_{MVn}$-structures $(A_1',1_1'),\dots,(A_n',1_n')$. Let $A$ 
be the $L_{MVn}$-structure $A_1\times \cdots \times 
A_n$ and $A'$ be the $L_{MVn}$-structure $A_1'\times \cdots \times A_n'$.  
\begin{Remark}\label{rkap69}
With the same notations, we deduce from Theorem \ref{th26} 
that in the language $L_{MVn}$ 
$(A,1_1,\dots,1_n)\equiv (A',1_1',\dots,1_n')\Leftrightarrow (A_1,1_1)\equiv (A_1',1_1'), 
\dots, (A_n,1_n)\equiv (A_n',1_n')$. 
\end{Remark} 
\indent  Now, we consider families of MV-chains, 
$(A_{1,\alpha_1},1_{1,\alpha_1})_{\alpha_1\in I_1},\dots, 
(A_{n,\alpha_n},1_{n,\alpha_n})_{\alpha_n\in I_n}$ 
(we can do the same thing with families of linearly ordered groups 
$(T_{1,\alpha_1})_{\alpha_1\in I_1},\dots, (T_{n,\alpha_n})_{\alpha_n\in I_n}$). 
For every $(\alpha_1,\dots,\alpha_n)$ in $I_1\times \cdots\times I_n$ we set 
$(A_{(\alpha_1,\dots,\alpha_n)},1_{1,\alpha_1},\dots,1_{n,\alpha_n})=
(A_{1,\alpha_1}\times \cdots \times A_{,n\alpha_n},1_{1,\alpha_1},\dots,1_{n,\alpha_n})$. 
We let 
$U$ be an ultrafilter on $I_1\times \cdots\times I_n$ and $(A,1_1,\dots,1_n)$ be the 
ultraproduct of the family $(A_{(\alpha_1,\dots,\alpha_n)},1_{1,\alpha_1},
\dots,1_{n,\alpha_n})$. We know that for $i\in \{1,\dots,n\}$, 
the canonical projection $p_i(U)$ on $I_i$ is an ultrafilter on $I_i$. Denote by $A_i$ the 
ultraproduct of the family $(A_{i,\alpha_i})$. Then one can prove that 
$(A,1_1,\dots,1_n)\simeq (A_1\times \cdots \times A_n,1_1,\dots,1_n)$, where $1_i$ is the 
greatest element of $A_i$. Note that the maximal element of $A$ is 
$1=1_1+\cdots +1_n$. \\[2mm]
\indent  In order to characterize some pseudofinite MV-algebras, 
we fix $n\in \NN^*$, and we restrict to the finite MV-algebras $A$ wich are isomorphic to 
a product of $n$ MV-chains $[0,1_1],\dots ,[0,1_n]$. In this case, saying that 
$A$ is hyperarchimedean is equivalent to saying that each of $[0,1_1],\dots ,[0,1_n]$ is simple. 
\begin{defi} We will say that an MV-algebra is $n$-{\it pseudofinite} 
if it is elementarily equivalent to some ultraproduct of a family of finite MV-algebras which 
are isomorphic to products of $n$ MV-chains. 
\end{defi}
\indent  We deduce the following. 
\begin{Theorem}\label{th611} 
Let $(A,1)$ and $(A',1')$ be MV-algebras. \\ 
1) $(A,1)$ is $n$-pseudofinite if, and only if, $A$ is projectable, it is isomorphic to a product 
of $n$ MV-chains $[0,1_1]\times\cdots\times [0,1_n]$ and, for every $i\in \{1,\dots,n\}$, 
the MV-chain $[0,1_i]$ is either finite or infinite discrete regular. \\ 
2) If $(A,1)$ and $(A',1')$ are $n$-pseudofinite, then $(A,1_1,\dots, 1_n)\equiv 
(A',1_1',\dots, 1_n')$ if, and only if, 
for every $i\in \{1,\dots,n\}$ either the MV-chains $[0,1_i]$, $[0,1_i']$ 
are finite and isomorphic, or they are infinite regular and satisfy the same formulas 
$D_{p^m,k}$. \end{Theorem}
\begin{proof} 1) $\Rightarrow$. Let  
$(A_{1,\alpha_1},1_{1,\alpha_1})_{\alpha_1\in I_1},\dots, 
(A_{n,\alpha_n},1_{n,\alpha_n})_{\alpha_n\in I_n}$ be families of MV-chains. 
For every $(\alpha_1,\dots,$ $\alpha_n)$ in $I_1\times \cdots\times I_n$ we set 
$(A_{(\alpha_1,\dots,\alpha_n)},1_{1,\alpha_1},\dots,1_{n,\alpha_n})=
(A_{1,\alpha_1}\times \cdots \times A_{,n\alpha_n},1_{1,\alpha_1},\dots,1_{n,\alpha_n})$. 
We let 
$U$ be an ultrafilter on $I_1\times \cdots\times I_n$ and for every $i$ let
$A_i$ be the 
ultraproduct of the family $(A_{i,\alpha_i})$ (associated with $p_i(U)$). 
If $A$ is the ultraproduct of the 
family $(A_{(\alpha_1,\dots,\alpha_n)},1_{1,\alpha_1},\dots,1_{n,\alpha_n})$, then 
$(A,1_1,\dots,1_n)\simeq (A_1\times \cdots \times A_n,1_1,\dots,1_n)$. Now, 
by Theorem \ref{th511} every 
$A_i$ is an MV-chain which is  either finite or infinite discrete regular. \\ 
\indent  $\Leftarrow$ If this holds, then in the language $L_{MVn}$ $A$ is isomorphic to 
$[0,1_1]\times \cdots \times [0,1_n]$. By Theorem \ref{th511}, every $[0,1_i]$ is 
isomorphic to an ultraproduct of a family of finite MV-chains 
$(A_{i,\alpha_i},1_{i,\alpha_i})_{\alpha_i\in I_i}$. Hence, $A$ is isomorphic to the ultraproduct 
of the family $(A_{1,\alpha_1}\times \cdots \times A_{,n\alpha_n},1_{1,\alpha_1},\dots,1_{n,\alpha_n})$. \\ 
\indent  2) follows from Remark \ref{rkap69} and Theorem \ref{th511}. 
\end{proof}
\indent  Now we turn to hyperarchimedean MV-algebras. By \cite[Corollary 6.5.6]{CDM}, 
being hyperarchimedean is equivalent to being a boolean product of simple MV-algebras. 
We will restrict ourselves to MV-algebras 
which are isomorphic to finite products of simple MV-algebras. One can prove that if 
it is hyperarchimedean, then every sub-MV-algebra is projectable, hyperarchimedean and is 
isomorphic to a finite product of simple MV-algebras. 
\begin{defi} We will say that an MV-algebra is $n$-{\it pseudo-hyperarchimedean} 
if it is elementarily equivalent to some ultraproduct of a family of hyperarchimedean MV-algebras 
which are  isomorphic to products of $n$ simple MV-algebras. 
\end{defi}
\indent  In the same way as Theorem \ref{th611}, and by using Theorem \ref{th511a}, 
one can prove the following. 
\begin{Theorem}\label{thm611}
 Let $(A,1)$ and $(A',1')$ be MV-algebras. \\ 
1) $(A,1)$ is $n$-pseudo-hyperarchimedean if, and only if, $A$ is projectable, 
it is isomorphic to a product 
of $n$ MV-chains $[0,1_1]\times\cdots\times [0,1_n]$ and, for every $i\in \{1,\dots,n\}$, 
the MV-chain $[0,1_i]$ is either finite or infinite and regular. \\ 
2) If $(A,1)$ and $(A',1')$ are $n$-pseudo-hyperarchimedean, then $(A,1_1,\dots, 1_n)\equiv 
(A',1_1',\dots, 1_n')$ if, and only if, 
for every $i\in \{1,\dots,n\}$ either the MV-chains $[0,1_i]$, $[0,1_i']$ 
are finite and isomorphic, or they are both discrete infinite regular and satisfy the same 
formulas $D_{p^m,k}$, or they are infinite dense regular and their Chang $\ell$-groups 
have the same prime invariants of Zakon. 
\end{Theorem}
\end{document}